\documentclass[a4paper,11pt,reqno,nocut]{article}

\usepackage[utf8]{inputenc}
\usepackage[T1]{fontenc}
\usepackage{amsmath,amssymb,amsthm,mathtools}
\usepackage{authblk}
\usepackage{mathrsfs,bm,color}
\usepackage{fancyhdr,lastpage}
\usepackage[shortlabels]{enumitem}
\usepackage{calc}
\usepackage{bbm}
\usepackage{graphicx}
\usepackage{upgreek}
\usepackage{lmodern}
\usepackage[normalem]{ulem}
\usepackage{verbatim}
\usepackage{bbm}
\usepackage{stmaryrd}
\usepackage{amsmath}
\usepackage{amssymb}
\usepackage{dsfont}
\usepackage{amsthm}
\usepackage{hyperref}
\hypersetup{
	colorlinks,
	linkcolor={red!50!black},
	citecolor={red!80!black},
	urlcolor={blue!80!black}
}

\usepackage{xcolor}

\usepackage{graphicx}

\usepackage{amsfonts}%
\usepackage{dsfont}

\allowdisplaybreaks
\usepackage[hmargin=2cm,vmargin=2cm]{geometry}

\numberwithin{equation}{section}

\usepackage{thmtools}
\declaretheorem[thmbox=M,name=Theorem,numberwithin=section]{theo}
\declaretheorem[name=Proposition,thmbox=M,numberwithin=section]{prop}
\declaretheorem[name=Lemma,thmbox=S,numberwithin=section]{lem}
\declaretheorem[name=Corollary,thmbox=M,numberwithin=section]{cor}

\theoremstyle{definition}

\newtheorem{rem}{Remark}[section]
\theoremstyle{plain}

\makeatletter
\renewenvironment{proof}[1][\proofname]{\par
	\pushQED{\qed}%
	\normalfont \topsep6\p@\@plus6\p@\relax
	\trivlist
	\item[\hskip\labelsep
	\sffamily\bfseries
	#1\@addpunct{.}]\ignorespaces
}{%
	\popQED\endtrivlist\@endpefalse
}
\makeatother

\newcommand{\jump}{{\vskip 0.3cm \noindent }}

\newcommand{\for}{\mbox{ for }}

\renewcommand{\and}{\mbox{ and }}

\newcommand{\N}{\mathbb{N}}

\newcommand{\R}{\mathbb{R}}

\newcommand{\Sn}{\mathbb{S}^{N-1}}

\newcommand{\dd}{\mathrm d}

\newcommand{\Ind}[1]{\1_{(#1)}}

\newcommand{\polclass}[1]{\ensuremath{\mathrm{#1}}}
\newcommand{\He}{\polclass{He}}
\newcommand{\Hek}{\polclass{He}_k}

\newcommand{\op}{\polclass{P}}
\newcommand{\opk}{\polclass{P}_k}

\newcommand{\E}{\mathbb{E}}
\renewcommand{\P}{\mathbb{P}}
\def\as{{ \mathrm{a.s.}  }}
\def\ed{\stackrel{{\mathcal{D}}}{=}}
\def\iid{\stackrel{{\mathrm{i.i.d}}}{\sim}}

\newcommand{\<}{\ensuremath{ \langle }}
\renewcommand{\>}{\ensuremath{ \rangle }}
\newcommand{\vect}[1]{\ensuremath{\boldsymbol{\mathbf{#1}}}}
\newcommand{\rdmvect}[1]{\ensuremath{\bm{#1}}}
\newcommand{\mat}[1]{\ensuremath{\boldsymbol{\mathbf{#1}}}}
\newcommand{\rdmmat}[1]{\ensuremath{\bm{#1}}}
\newcommand{\diag}{\ensuremath{\mat{D}}}
\newcommand{\Tr}{\ensuremath{\mathrm{Tr}}}

\newcommand{\rank}{\ensuremath{\mathrm{rank}}}
\newcommand{\norm}[1]{\ensuremath{ \|#1 \|}}
\newcommand{\normop}[1]{\ensuremath{ \|#1 \|_{\mathrm{op}} }}

\newcommand{\normf}[1]{\ensuremath{ \|#1 \|_{\mathrm{F}} }}

\def\ks{{k_{\star}}}
\def\Yf{{ \rdmmat{Y}^{(f)}  }}

\def\coeffk{{\vartheta_k(f)}}

\def\varfZ{{\vartheta_0(f^2) - \vartheta_0(f)^2 }}
\def\coeffkn{{\vartheta_k(f_t)}}

\def\xk{{\vect{x}^k}}

\newcommand{\pert}{\mathrm{pert}}
\def\fm{{f_M}}
\def\fn{{f_t}}
\def\fdm{{f_{\delta,M}}}

\def\fnp{{f_{t}^{\pert}}}

\def\fnpt{{\Tilde{f}_{t}^{\pert} }}
\def\Ytf{{\rdmmat{Y}_{L}^{(f_L)} }}

\def\Ytfnp{{\rdmmat{Y}_{L}^{(\fnp)} }}

\def\wkZ{{w^{(k)}_Z}}

\def\coeffkfdm{{\vartheta_k(\fdm)}}

\def\coeffkfnp{{\vartheta_k(\fnp)}}
\def\coeffksfnp{{\vartheta_{\ks}(\fnp)}}

\def\Lpdx{{L^p(\dd x)}}
\def\Ltwodx{{L^2(\dd x)}}
\def\Lfourdx{{L^4(\dd x)}}

\def\LoneZ{{L^1(\mu_Z)}}
\def\LtwoZ{{L^2(\mu_Z)}}
\def\LfourZ{{L^4(\mu_Z)}}

\DeclareMathOperator{\1}{\mathbbm{1}}

\renewcommand{\P}{\mathbb{P}}

\renewcommand{\leq}{\leqslant}
\renewcommand{\geq}{\geqslant}
\renewcommand{\epsilon}{\varepsilon}

\def\supp{{\mbox{supp}}}
\def\pP{{\mathbb P}}

\def\R{{\mathbb R}}
\def\rank{{\rm rank}}

\begin{document}

\title{Spectral Phase Transitions in Non-Linear 
 Wigner Spiked Models}

\author{Alice Guionnet\thanks{UMPA, ENS Lyon and CNRS, France. Email: \texttt{aguionnet@ens-lyon.fr}}
	, 
	Justin Ko\thanks{UMPA, ENS Lyon and CNRS, France and Department of Statistics and Actuarial Science, University of Waterloo, Canada. Email: \texttt{justin.ko@uwaterloo.ca}},
	Florent Krzakala\thanks{IdePHICS laboratory, \'Ecole F\'ed\'erale Polytechnique de Lausanne, Switzerland. Email: \texttt{florent.krzakala@epfl.ch}}
	,
	Pierre Mergny\thanks{IdePHICS laboratory, \'Ecole F\'ed\'erale Polytechnique de Lausanne, Switzerland. Email: \texttt{pierre.mergny@epfl.ch}}
	,
	Lenka Zdeborov\'a\thanks{SPOC laboratory,  \'Ecole F\'ed\'erale Polytechnique de Lausanne, Switzerland. Email: \texttt{lenka.zdeborova@epfl.ch}} }

\date{}

\maketitle
\begin{abstract}
We study the asymptotic behavior of the spectrum of a random matrix where a non-linearity is applied entry-wise to a Wigner matrix perturbed by a rank-one spike with independent and identically distributed entries. In this setting, we show that when the signal-to-noise ratio scale as $N^{\frac 12 (1-1/\ks)}$, where $\ks$ is the first non-zero generalized information coefficient of the function, the non-linear spike model effectively behaves as an equivalent spiked Wigner matrix, where the former spike before the non-linearity is now raised to a power $\ks$. This allows us to study the phase transition of the leading eigenvalues, generalizing part of the work of Baik, Ben Arous and Pech\'e to these non-linear models. 
\end{abstract}
\tableofcontents


\newpage 

\section{Introduction and Notations}
\label{sec:intro}

\subsection{Introduction}%
\label{sec:introRMT-BBP-NLM}
Random Matrix Theory (RMT)\cite{AGZ,mehta2004random,potters2020first,tao2023topics} has its roots in Wishart's 1928 statistical investigations \cite{Wishart1928} and Wigner's 1950s work on nuclear models \cite{Wigner1958OnTD}. Since then, its influence has spread to a variety of fields, including high-energy physics \cite{itzykson_planar_2008,Verbaarschot72}, spin glass models \cite{marinari_replica_1994,auffinger_random_2013}, and number theory \cite{keating_random_2000}. At its core, RMT delves into the intricate high-dimensional and spectral nuances of select random matrices. Notably, seminal RMT insights determined that the spectrum of matrices first explored by Wishart and Wigner aligns with the semi-circular law and the Mar\v{c}enko-Pastur distribution, respectively. 

While Wigner and Wishart matrices can be considered as pure noise random matrices, one of the most studied models beyond these two examples is the  \emph{spiked models}, in which a fixed rank matrix, playing the role of a signal, is added to the former Wigner or Wishart matrix. In its simplest form, for example, when one studies the rank-one perturbation of a Wigner matrix (see Sec.~\ref{sec:usualBBP} for the quantitative details), one can show \cite{BBP,paul_asymptotics_2007,baik_eigenvalues_2006,renfrew13,Pch2005} that in the high-dimensional regime, the highest eigenvalue in absolute value $\lambda_1$ of this new random matrix undergoes a phase transition, depending on the strength $\gamma$ of the rank-one perturbation, from a regime where it sticks to the edge of the limiting spectral distribution, to the one where it pops out of the bulk. Similarly, the associated eigenvector behaves as a random vector uniformly sampled on the sphere before the transition and becomes partially aligned with the vector of the rank-one perturbation after the transition. 

This phenomenon is now commonly referred to as the `BBP phase transition' after the names of the three authors \cite{ BBP} who studied fluctuations around the deterministic limit for this type of model. This BBP phase transition phenomenon has been generalized and applied to a variety of different settings; in particular, we refer the reader to \cite{Capitaine12fluctuations,florent_benaych-georges_fluctuations_2011,bloemendal_limits_2013} for other studies of the fluctuations in similar spike models, to \cite{Maida2007LDP,biroli2019LDP} for the study of the large deviation of the top eigenvalue and eigenvector whenever the noise is Gaussian, to \cite{benaych-georges_eigenvalues_2011,benaych-georges_singular_2012} for finite-rank perturbation of unitary invariant matrices and to \cite{lesieur2015mmse,dia2016mutual,krzakala2016mutual,lesieur2016phase,lelarge2017fundamental,lenkarmt,banks2018information} for the study of these spike models from the information-theoretic point of view.

A second family of models that has been extensively studied in the past decade due to its connection with high-dimensional statistical problems when both the number of samples and the dimension diverge is the family of \emph{ nonlinear (random matrix) models}, where one study the property of an \emph{entry-wise} non-linearity function $f(\cdot)$ applied to a given random matrix, and we refer in particular to the works \cite{ElKaroui10Kernel,ElKarouiSpikeKernel,Benaych16Gram,romain_couillet_kernel_2016,goldt2022gaussian} for their relation to kernel methods and to \cite{pennington_nonlinear_2017,Peche19PenningtonWorah,mei2022generalization,gerace2020generalisation,loureiro2021learning,piccolo21,Louart18RMTtoNN,OptimalityPCA}
for the relation to the spectrum of one-layer neural network at initialization (or random features).

This paper aims to consider a mixture of these two important families of models by establishing the spectral properties of non-linear matrix models applied to rank-one perturbation of Wigner matrices
\begin{align}
    Y^{(f)}_{ij} 
    &:=
     \frac{1}{\sqrt{N}} \bigg[  f \left( Z_{ij} + \frac{\gamma}{\sqrt{N}} x_i x_j \right) - \E_Z f(Z) \bigg] 
    , &&\label{eq:setting}
\end{align}
for ‘arbitrary' non-linearity functions $f$, and ‘arbitrary' distributions of $Z_{ij}$ and $\vect{x} = (x_1,\dots, x_N)$.
Studying the spectrum of these matrices, we will see that the most interesting regime is when 
 the \emph{signal-to-noise ratio} (SNR) $\gamma$ scaled with the system size $N$. 

Matrices of the form (\ref{eq:setting}) are also important in theoretical machine learning. They can be seen as kernel methods applied to spiked models, but they also appear in the studies of gradient descent in deep neural networks where $Z$ corresponds to the pre-activations at initialization, $f$ is the non-linear activation function, and the low-rank perturbation originates from the dynamics of the early steps of training with gradient descent such as in \cite{ba2022high,damian2022neural,dandi2023learning}. 

A special case of the problem we study was considered by \cite{9517881}, who studied the information-theoretic aspects of recovering a low-rank signal $x$ from unsigned observations of the matrix, corresponding to $f(x) = |x|$, but they have not studied the spectrum of such matrices. 

Our main goal is to study the asymptotic behavior of the leading eigenvalue of the matrix $Y^{(f)}$, namely the largest eigenvalue in absolute value
of the matrix $Y^{(f)}$: depending on the sign of $\gamma$ and $f$, it can be the smallest or the largest eigenvalue of $Y^{(f)}$. 

\subsection{Motivating Examples}
\label{sec:examples}
As an empirical motivation of the non-trivial behavior one can obtain in this setting, let's consider $x_i$ and $Z_{ij}$  standard independent Gaussian random variables and construct the two $(N \times N)$ symmetric matrices $\rdmmat{Y}^{(f_1)} := {(Y^{(f_1)}_{ij})}_{1\leq i,j \leq N}$  and $\rdmmat{Y}^{(f_2)} :={(Y^{(f_2)}_{ij})}_{1\leq i,j \leq N}$ with
\begin{align}
\label{eq:heuristic_ex}
    Y^{(f_a)}_{ij} 
    &:=
     \frac{1}{\sqrt{N}} \bigg[  f_a \left( Z_{ij} + \frac{\gamma}{\sqrt{N}} x_i x_j \right) - \E_Z f_a(Z) \bigg] 
     \quad
     \for 
     a=1,2
     \, , &&
\end{align}
with $f_1(x) := |x|$ and $f_2(x) :=x^3- 3 x$ (one has $\E_Z f_1(Z) = \sqrt{2/\pi}$ and $\E_Z f_2(Z) = 0$ in these cases). In the leftmost plot of Fig.~\ref{fig:abs_rescaled} and Fig.~\ref{fig:He3_rescaled}, we have plotted the empirical position of the top eigenvalue $\lambda_1$ of each matrix, as a function of $\gamma$, for different values of the size $N$. Unlike the well-known result for the usual ‘linear' setting ($f(x):=x$), as one increases the dimension $N$, one observes a \emph{shift towards the right} of the curves. This empirical phenomenon suggests rescaling the constant $\gamma \equiv \gamma(N)$ with $N$ to get a deterministic limit in the large $N$ limit. An empirical fit further suggests a scaling of the form $\gamma = N^{\alpha}$ with $\alpha \approx 0.25$ for  $\rdmmat{Y}^{(f_1)}$ and $\alpha \approx 0.33$ for  $\rdmmat{Y}^{(f_2)}$ and a natural question is to give a theoretical explanation for these scaling.

\begin{figure}[t]
    \centering
    \includegraphics[width=0.325\linewidth]{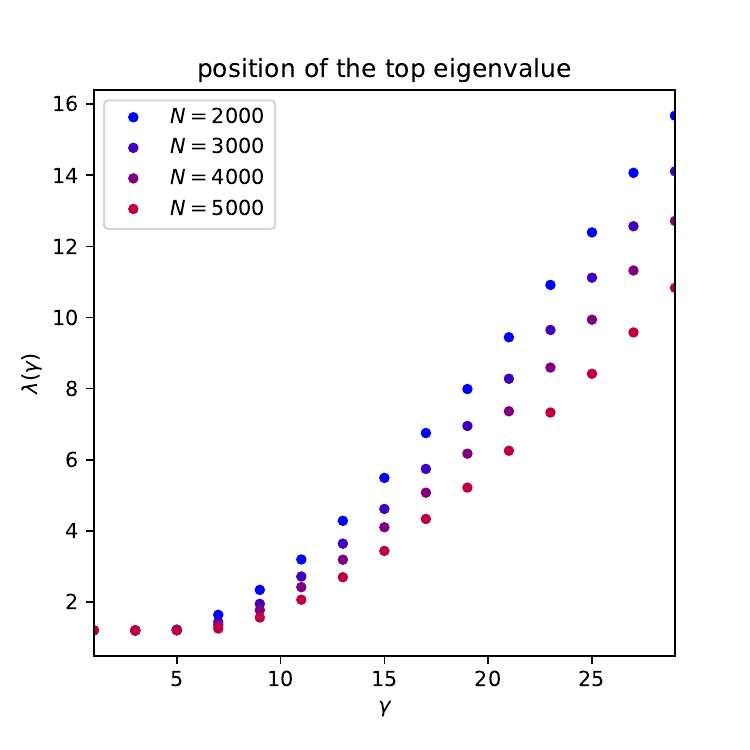}
    \includegraphics[width=0.325\linewidth]{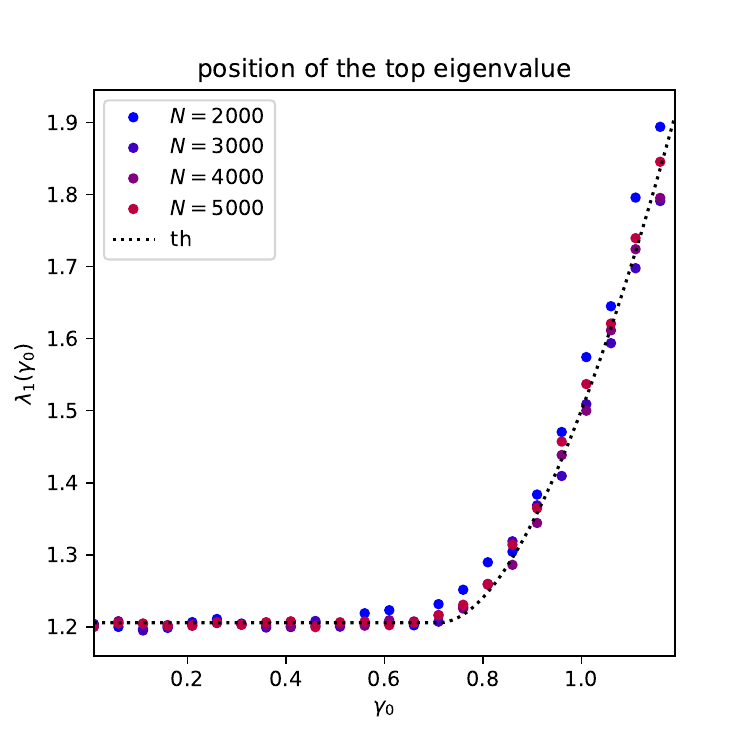}
    \includegraphics[width=0.325\linewidth]{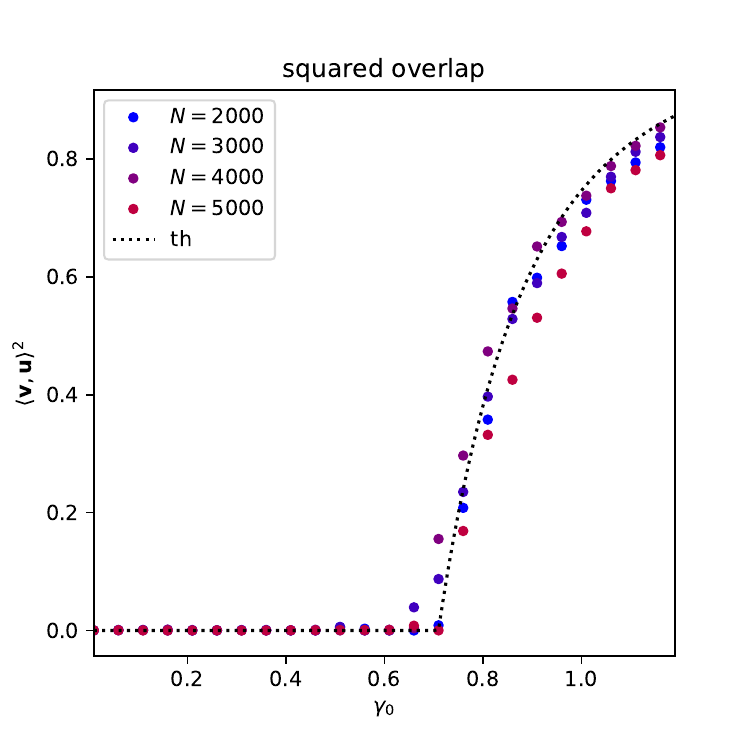}
      \caption{Illustration of the main results for the matrix model defined by Eq.~\eqref{eq:heuristic_ex} with non-linearity $f_1(x)=|x|$ and Gaussian noise for which the information index is $\ks=2$. \textbf{(Left)} The position of the leading eigenvalue as a function of the SNR $\gamma$ for different matrix sizes $N$. As $N$ increases, one observes a shift of the curves towards the right, suggesting that one should rescale the SNR with $N$. \textbf{(Middle)} The same data plotted against the rescaled SNR $\gamma_0 = \gamma N^{-1/4}$ given by Eq.~\eqref{eq:nontrivialscaling_intro}. The dotted black curves represent the theoretical values given by a BBP-like theorem Eq.~\eqref{eq:BBPNL:leadingeig_intro}. \textbf{(Right)} 
       The squared overlap of the leading eigenvector with the vector $ \vect{u} \equiv  \vect{x}^{\ks}/\| \vect{x}^{\ks}\|$ where $\vect{x}$ is the signal vector as a function of the \emph{rescaled} SNR $\gamma_0$ given by Eq.~\eqref{eq:nontrivialscaling_intro} and different sizes $N$. The dotted black curves represent the theoretical values given by a BBP-like theorem  Eq.~\eqref{eq:BBPNL:overlap_intro}. }
    \label{fig:abs_rescaled}
\end{figure}

\begin{figure}[t]
    \centering
    \includegraphics[width=0.32\linewidth]{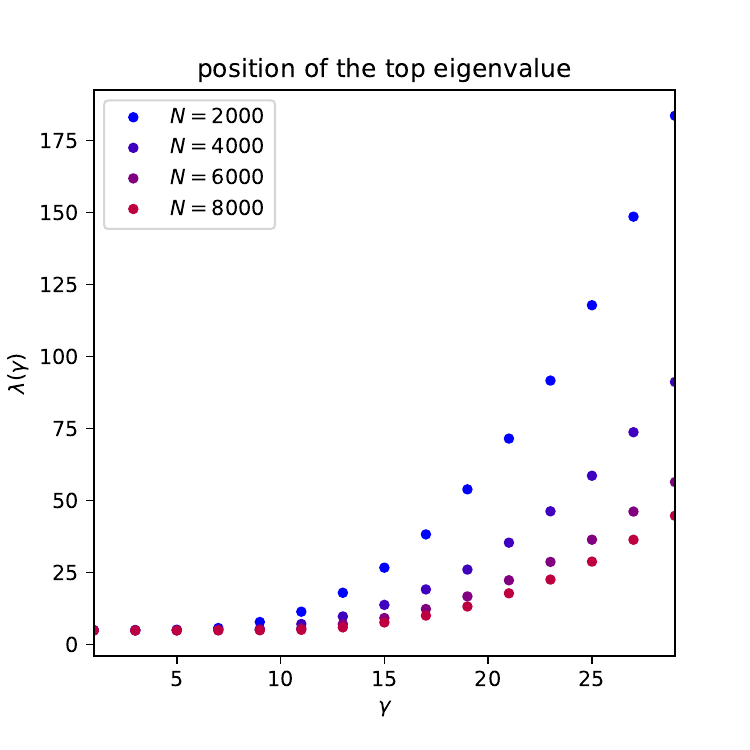}
    \includegraphics[width=0.32\linewidth]{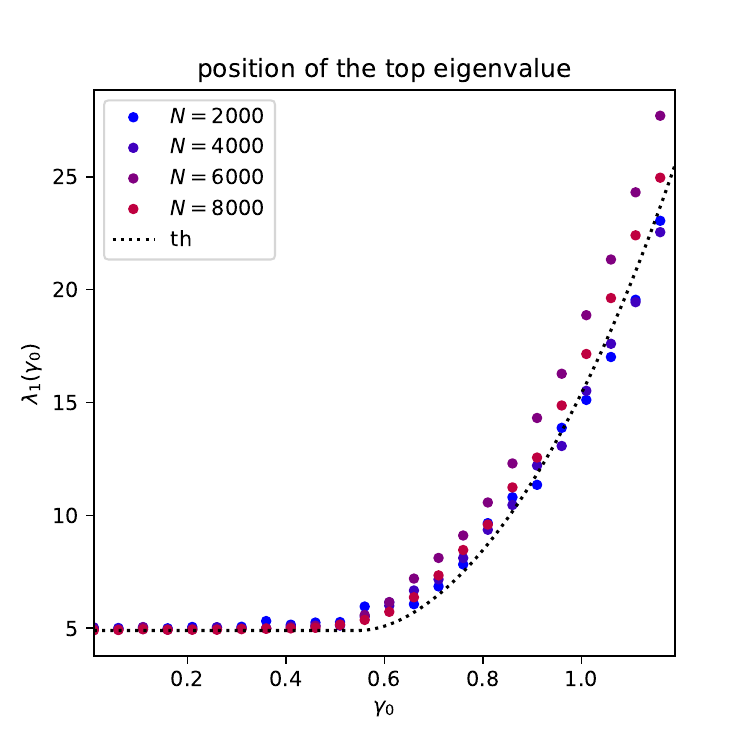}
    \includegraphics[width=0.32\linewidth]{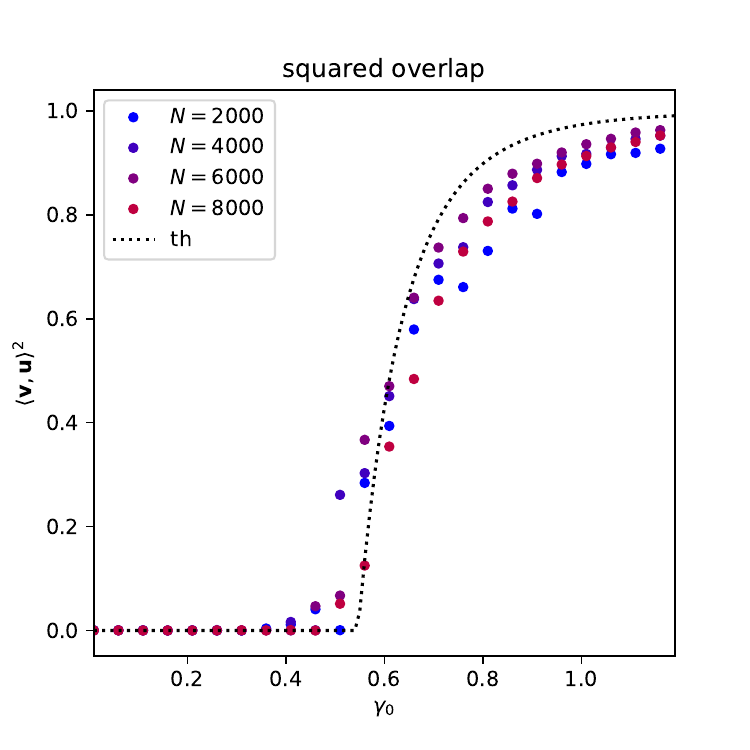}
       \caption{The same as Fig.~\ref{fig:abs_rescaled} for the non-linearity $f_2(x)=x^3 - 3x$ and Gaussian noise for which the information index $\ks=3$ leading to a rescaling of the SNR $\gamma_0 = \gamma N^{-1/3}$.      
       }
    \label{fig:He3_rescaled}
\end{figure}

\subsection{Summary of the Main Results}
\label{sec:maincontributions}

We now give a broad overview of the main results obtained in this work and explain quantitatively the behavior of leading eigenvalues and eigenvectors observed in the motivating examples in a more general framework. A careful statement of these main results with precise conditions on the assumptions will be stated in Section~\ref{sec:MainResults}. Motivated by the examples above, we seek to understand what happens to the spectrum of a matrix
\begin{align}
    Y^{(f)}_{ij} 
    &:=
     \frac{1}{\sqrt{N}} \bigg[  f \left( Z_{ij} + \frac{\gamma}{\sqrt{N}} x_i x_j \right) - \E_Z f(Z) \bigg] 
    , &&
\end{align}
for ‘arbitrary' non-linearity functions $f$, and ‘arbitrary' distributions of $Z_{ij}$ and $\vect{x} = (x_1,\dots, x_N)$. For the sake of simplicity, suppose that the law of $Z_{ij}$ has a smooth density function $w_Z$ that vanishes at infinity, and that the entries of $\vect{x}$ are independent with exponential decay. Both of these assumptions are weakened and stated in more generality in Section~\ref{sec:MainResults}, but the statements of the following results become slightly less explicit without densities. 

The relevant scaling for the SNR $\gamma$ at which the leading eigenvalue is of order one depends on certain statistics that build on the non-linearity $f$ and the law of the noise $w_Z$. We define the  \emph{$k$-th information coefficient} associated with the non-linearity $f$ and density $w_Z$ of $Z$ by
\begin{align}
    \label{eq:informationcoeff_smoothdensity_intro}
    \vartheta_k(f) &:= (-1)^{k} \int_{\mathbb{R}} f(x) w_Z^{(k)}(x) \dd x \, . &&
\end{align}
where $w_Z^{(k)}$ is the $k$-th derivative of the density. 
We then define the associated \emph{information index} $\ks$ by the first non-zero component of the information coefficient as 
\begin{align}
	\label{eq:def_kstar_intro}
	\ks 
    &:=
    \inf \left\{ 1 \leq k \leq k_0  \, \mbox{ such that } \,  \coeffk \neq 0 \right\} 
    \, . &&
 \end{align}

Whenever the entries $Z_{ij}$ are standard Gaussian random variables, the generalized information coefficients $\vartheta_k(f)$ appearing in this paper reduce to the classical \emph{Hermite coefficients}, that is, the scalar product of $f$ with the $k$-th Hermite polynomial in the $L^2$ space weighted by the Gaussian measure. 
These Hermite coefficients appear naturally in various high-dimensional learning problems, for example \cite{arous2021online}, where the associated critical index $\ks$ dictates in a certain manner the \emph{rate} of the learning, in a similar fashion as the relevant scaling of the SNR in this problem, depends on $\ks$. In Ref.~\cite{arous2021online}, these coefficients are called ‘information coefficients' rather than ‘Hermite coefficients', hence the name ‘generalized information coefficients' used here to denote the more general framework where the noise is not necessarily Gaussian. It will be interesting to know if one can reinterpret these results in terms of the spectral properties of a non-linear rank-one perturbation model.

A consequence of the first main result of our paper indicates that the relevant scaling $\alpha$ of the \emph{signal-to-noise ratio} is given by 
\begin{align}
    \gamma(N) = \gamma_0 N^{\frac{1}{2} (1 - \frac{1}{\ks})} \,. &&
\end{align}
Furthermore, under this scaling and for large $N$, we show that the non-linear model studied in this paper behaves, up to a vanishing error term, as the usual ‘Wigner matrix plus rank-one perturbation' model, with a ‘new' rank-one perturbation aligned with the vector $\vect{x}^{\ks}=(x_1^\ks,\dots,x_N^\ks)$.

\begin{theo}[Rank-one Equivalence for the Non-Linear Model (Informal)]\label{thm:rank1_informal}
Suppose that $Z$ has a smooth density $w_Z$ vanishing exponentially fast at infinity and that $\vect{x}$ has independent entries with exponential tails. Assume that $f$ is a measurable function such that $\E [f(Z)^4] < \infty$. Moreover, assume that we are in the scaling regime for the SNR where there exists a positive constant $\gamma_0$ such that
\begin{align}\label{eq:nontrivialscaling_intro}
\gamma(N) = \gamma_0 N^{\frac{1}{2} (1 - \frac{1}{\ks})} \,.&&
\end{align}
Then for all $N \geq 1$, we have
\begin{align}
	\Yf 
    &\ed 
    {\small{\sqrt{\vartheta_0(f^2) - \vartheta_0(f)^2} 
    \,}} \rdmmat{W}+ \gamma_0
    \frac{  \vartheta_{\ks}(f) }{\ks !}\bigg[  \frac{\vect{x}^{\ks}}{\sqrt{N}} \bigg(\frac{ \vect{x}^{\ks}}{\sqrt{N}}\bigg)^\intercal  \bigg] + \rdmmat{E} 
    \, , &&
\end{align}
where $\rdmmat{W}$ is a \emph{centered and normalized Wigner matrix with entries with bounded fourth  moment} and the operator norm of $\rdmmat{E}$ goes to $0$ as $N \to \infty$. Moreover, $\rdmmat{W}$ and $\vect{x}$ are independent. 
\end{theo}
The rank-one equivalence at the relevant scaling allows us to obtain the limiting behavior of both the leading eigenvalue and associated eigenvector of this non-linear model through classical methods. In particular, the matrix $\Yf$ exhibits the usual BBP transition \cite{BBP}.
\begin{cor}[The BBP Transition (Informal)]\label{thm:BBP_nonlinear_informal}
Under the same Hypothesis of Theorem~\ref{thm:rank1_informal}, if we denote by $\lambda_1$ and $\rdmvect{v}_1$ the leading eigenvalue and eigenvector of $\Yf$, we have:
\begin{align}
\label{eq:BBPNL:leadingeig_intro}
    \lambda_1 
    &\xrightarrow[N\to \infty]{\as} 
    \mathsf{l} \left( \gamma_0^{\ks} \frac{ \vartheta_\ks(f)}{\ks !} \, m_{2\ks} \, , \, \sqrt{\vartheta_0(f^2) - \vartheta_0(f)^2} \right)
    \, , && 
\end{align}
and
\begin{align}
\label{eq:BBPNL:overlap_intro}
    \Bigg\langle \rdmvect{v}_1, \frac{\vect{x}^{\ks}}{\| \vect{x}^{\ks}\|} \Bigg\rangle^2 
    &\xrightarrow[N\to \infty]{\as}
    \mathsf{m}\left( \gamma_0^{\ks} \frac{ \vartheta_\ks(f)}{\ks !} \, m_{2\ks} \, , \, \sqrt{\vartheta_0(f^2) - \vartheta_0(f)^2} \right)  
    \, . && 
\end{align}
where $\mathsf{l}(.,.)$ and $\mathsf{m}(.,.)$ are the classical quantities for the leading eigenvalue and squared overlap for rank-one perturbations of Wigner matrices as defined in Eq.\ \eqref{eq:def:leadeig_BBP}.
\end{cor}

Putting these results into action for the Gaussian noise and the non-linear functions $f_1(x) := |x|$ and $f_2(x) := x^3- 3x $ in Sec.~\ref{sec:examples}, one has respectively $\ks = 2$ and $\ks=3$ such that the relevant scaling of Eq.~\eqref{eq:nontrivialscaling} writes $\gamma(N) = \gamma_0 N^{\alpha}$, with $\alpha=1/4$ and $\alpha=1/3$ respectively, in accordance with the scaling of the SNR found numerically. In the middle plots of Fig.~\ref{fig:abs_rescaled} and Fig.~\ref{fig:He3_rescaled}, we then plot the leading eigenvalues as a function of the rescaled SNR $\gamma_0$ to see that the empirical data points collapse to the theoretical curve from the BBP transition of Corrollary \ref{thm:BBP_nonlinear}. In the rightmost plots of Fig.~\ref{fig:abs_rescaled} and Fig.~\ref{fig:He3_rescaled}, we then plot the overlap between the leading eigenvector and the signal again the rescaled SNR $\gamma_0$, again seeing a collapse of the theoretical prediction of the BBP transition of Corollary \ref{thm:BBP_nonlinear} with again a perfect agreement.

\jump

A major advantage of our approach is that our decomposition is quite robust, and our result can be trivially extended to deal with variants and generalizations of such a model, such as finite rank deformations of Wigner matrices, rank-one deformations of variance-profile Wigner matrices, and rank-one deformations of rectangular matrices (see Section~\ref{sec:variants_and_gen}).

\paragraph{Organization of the paper -}
\label{sec:organization}
The rest of this article is organized as follows: in Sec.~\ref{sec:MainResults}, we introduce the main non-linear model studied in this work and define formally the important coefficients $\coeffk$. We then introduce the main result of our paper, summarized in Theorem \ref{thm:rank1} from which Theorem~\ref{thm:BBP_nonlinear_informal} and Corollary~\ref{thm:rank1_informal} are merely special cases under stronger assumptions. Next, in Sec.~\ref{sec:variants_and_gen}, we  generalize our result to higher-rank perturbations and non-symmetric matrices. 

The remaining sections of the paper are devoted to the proof of our result. We give a broad overview of the proof here, and defer to Sec.~\ref{sec:outlineproof} for a detailed overview of the proof strategy.  Our approach uses a rank-one equivalence of the non-linear model to obtain the spectral phase transitions, in comparison to other approaches which involve computing moments or using Stieltjes transforms. From the rank-one equivalence, we can read off the critical scaling as a function of $\ks$ and apply classical results for perturbations of Wigner matrices to obtain the phase transitions at this critical scaling. Furthermore, the proof of the rank-one equivalence is given entrywise, so it can be easily adapted to prove similar equivalences in for higher rank and non-symmetric matrices as stated in Sec.~\ref{sec:variants_and_gen}.
	
The rank-one equivalence is first proved in a simpler case when the non-linearity is smooth, in Sec.~\ref{sec:proof:smooth} through an application of Taylor's theorem and generic bounds on the operator norms of random matrices with independent entries to control the error. In Sec.~\ref{sec:proof:notsmooth}, we prove that the results extend to non-smooth non-linearities provided that there exists a sequence of smooth approximating functions which preserves the information index. Lastly, in Sec.~\ref{sec:proofsmoothdensity}, we show that when the noise has a smooth density function, we can explicitly construct this sequence of index preserving smooth approximating functions through a truncation and smoothing argument, showing that we can relax the assumptions on the function $f$.

\subsection{Notations}
\label{sec:Notations}
 We recall that a sequence of events $(E_N)_{N\in \N}$ occurs \textbf{with overwhelming probability} if for any fixed $A>0$, there exists a constant $C_A$ independent of $N$ such that   $\P(E_N) \geq 1 - C_A N^{-A}$ for every $N \in \N$.   Similarly, a sequence of events $(E_N)_{N\in \N}$ occurs \textbf{with high probability} if for some $c>0$ independent of $N$, there exists constants $c,C>0$ independent of $N$ such that  $\P(E_N) \geq 1 - C N^{-c}$. We say that the law $\pP(\cdot)$ of a random variable $X$ has \textbf{stretched exponential tails} if there exists a constant $\alpha > 0$ and universal constants $C,c>0$ such that for every $M>0$, 
$\pP(|X| > M ) \leq C e^{-cM^\alpha}$.  We call $\alpha$ an exponent of the stretched exponential decay. Equality in law between two random variables $A$ and $B$ is denoted by $A \ed B$.
\jump 
A function $f$ is \textbf{locally Lipschitz with polynomial constant} if for every $L > 0$, there exists a constant $C(L)$ such that
\begin{align}
\label{eq:def:loclip}
   \sup_{|x|, |y| \leq L} \bigg| \frac{ f(x) - f(y) }{x - y} \bigg| 
    &\leq
    C(L)
    , &&
\end{align}
 and $C(L)$ is bounded  by a polynomial, that is there exists a $\beta \in \N$ and a finite constant $C$ such that for every real number $L$, $|C(L)| \leq C |L|^\beta$. We recall that if $C(L) \equiv C$ is independent of $L$ and the supremum is over the set of definition of $f$, then $f$ is \textbf{Lipschitz}.
 \jump 
 For $\vect{x}:=(x_1,\dots, x_N)$ and $k$ an integer, we denote by $\xk := (x_1^k,\dots, x_N^k)$ the vector where each entry is raised to the power $k$. For $(N\times M)$ matrices (or vectors if $M=1$) $\mat{A}$ and $\mat{B}$, the \textbf{Hadamard product} is denoted by $(\mat{A} \odot \mat{B})_{ij} := A_{ij} B_{ij}$ and the \textbf{Hadamard  power} by $(\mat{A}^{\odot k})_{ij} := A_{ij}^k$. The \textbf{leading eigenvalue and eigenvectors} of a symmetric matrix, refer to the eigenvalue  with the largest absolute value and its associated eigenvector. 

\section{Main Results}
\label{sec:MainResults}

\subsection{Model and Assumptions for the Non-Linear Rank-One Perturbation}
\label{sec:model_and_assumptions}
We consider the symmetric matrix $\Yf$ with entries given by
\begin{align}
\label{eq:model_entries}
    Y_{ij} 
    &:= 
    \frac{1}{\sqrt{N}} \bigg[  f \left( Z_{ij} + \frac{\gamma(N)}{\sqrt{N}} x_i x_j \right) - \E f(Z_{ij}) \bigg] 
     \, , &&
\end{align}
where $\vect{x} = (x_1, \dots, x_N)^\top$ is the \emph{signal vector}, $\rdmmat{Z} = \left( Z_{ij} \right)_{1 \leq i,j \leq N}$ is a symmetric \emph{noise matrix}, $f$ is the \emph{non-linearity} and $\gamma(N)$ is the \emph{signal-to-noise ratio} (SNR) which, in this work, might depend on the dimension $N$. Subtracting the constant terms $ \E f(Z_{ij}) $ avoids the appearance of a (non-informative) Perron-Frobenius mode. We now list the assumptions needed for our main result to hold, where some of them involve an integer $\ks$ defined in the next section.

\jump
\underline{\emph{For the signal vector:}}
\noindent
\newcounter{hypcount}
\newcounter{hyppcount}
\begin{enumerate}[leftmargin=1.5cm,label={(\bfseries H\arabic*)}]
\item\label{hyp:x} 
We assume the entries of the vector $\vect{x}$ to be independent and identically distributed (iid), that is with law $ \pi_X^{\otimes N} $ and the distribution $\pi_X(\cdot)$ has \emph{stretched exponential tails} and is assumed to be independent of the dimension $N$. For simplicity, we also assume the distribution to be non-trivial, that is $\pi_{X} \neq \delta_0$.
\setcounter{hypcount}{\value{enumi}}
\end{enumerate}
We also denote by $m_{k} := \E_{x \sim {\pi_X}} [ x^{k} ]$, the $k$-th moment associated to this distribution. 

\jump
\underline{\emph{For the noise matrix:}}
\noindent
\begin{enumerate}[leftmargin=1.5cm,label={(\bfseries H\arabic*)}]
\setcounter{enumi}{\value{hypcount}}
\item\label{hyp:noise} 
We assume the entries $Z_{ij}$ are $iid$ and with distribution $\mu_Z$ for $i \leq j$, where $\mu_Z$ has \emph{stretched exponential tails} and is independent of the dimension $N$. 
\setcounter{hypcount}{\value{enumi}}
\end{enumerate}
We denote by $I$ the support of the distribution of $Z$ and by $\left\{ \opk \right\}_{k \in \N}$ the family of orthogonal polynomials associated to $\mu_Z$, which always exists since $\mu_Z$ admits moments of any order. The case where $\dd \mu_Z(x) = \dd \mu_G(x):= w_G(x) \dd x :=  (2 \pi)^{-1/2} \mathrm{e}^{-x^2/2} \dd x$ is the Gaussian distribution is of particular importance and in this case, we simply denote  $Z \equiv G$. We recall that for the Gaussian measure, the associated orthogonal polynomials are the famous \emph{Hermite polynomials} $\mathrm{He}_{k}(x)$ defined as the solution of the differential-recurrence relation $\mathrm{He}_{n + 1} = x\mathrm{He}_{n} - \mathrm{He}_{n}^{\prime}$ with initial conditions $\He_0 =0$ and $\He_1 =1$. 

\jump
\underline{\emph{For the non-linearity:}}
\jump 
For a function $f \in C^{k_0}$  for some $k_0 \geq 1$ and $f^{(k)} \in L^1(\dd \mu_Z)$ for  $k \leq k_0$, we define  for $k \leq k_0$ the \textbf{$k$-th (generalized) information coefficient} of $f$ as      
\begin{align}
        \label{eq:coeff_f_smoothf}
          \coeffk
          &= 
          \E_Z f^{(k)}(Z)
          \, ; &&
    \end{align}
When $k=0$, this coefficient simplifies to $\vartheta_0(f) = \E_Z f(Z)$ and in particular  $\vartheta_0(f^2) - \vartheta_0(f)^2 $ is the variance of the random variable $f(Z)$ for $f \in L^2(\dd \mu_Z)$. We define the \textbf{(generalized) information index} $\ks$ as:
\begin{align}
	\label{eq:def_kstar}
	\ks 
    &:=
    \inf \left\{ 1 \leq k \leq k_0  \, \mbox{ such that } \,  \coeffk \neq 0 \right\} 
    \, . &&
 \end{align}
\jump
Next, we either assume that $f$ \emph{is sufficiently smooth and bounded} in the following sense:
 \begin{enumerate}[leftmargin=1.5cm,label={(\bfseries{$\mathbf{\tilde{H}}$}\arabic*)}]
\setcounter{enumi}{\value{hypcount}}
\item\label{hyp:proof:ssmoothf} $f \in C^{\ks+1}$ such that $f^{(k)} \in L^4(\mu_Z)$ for $k \leq \ks$ and $\norm{f^{(\ks+1)}}_{\infty} < \infty$.
\end{enumerate}
or in the case that $f$ is not smooth, that
\begin{enumerate}[leftmargin=1.5cm,label={(\bfseries H\arabic*)}] 
\setcounter{enumi}{\value{hypcount}}
\item \label{hyp:nonlin_generalized:smoothapprox}
\begin{enumerate}
    \item\label{hyp:nonlin_generalized}  $f \in L^4(\mu_Z)$ and $f$ is locally Lipschitz with polynomial  constant,
\end{enumerate}
\jump 
and we can find a sequence $(f_t)_t$ of smooth functions $f_t \in \LfourZ$ such that 
    \begin{enumerate}
    \setcounter{enumii}{1}
    \item $f_{t} \xrightarrow[t\to \infty]{}f$ in $\LfourZ$. \label{hyp:nonlin_generalized:L4}
    \item $f_{t} \in C^\infty$ has compact support. \label{hyp:nonlin_generalized:cptsup}
    \item There exists $\ks$ such that $f_t^{(k)} \in \LfourZ$ for $0 \leq k \leq \ks$, $\| f_t^{(\ks + 1)} \|_{\infty} < \infty$ for all $t$ and
    \label{hyp:nonlin_generalized:derivbounds}
    \item $\vartheta_{k}(f_{t}) \xrightarrow[t\to \infty]{}0$ for all $k < \ks$ and $\vartheta_{\ks}(f_{t})$ converges as $t$ goes to infinity towards  a non-vanishing quantity that we denote $ \vartheta_{\ks}(f)$. In this case, we define the information index of $f$ to be this index $\ks$.     
     \label{hyp:nonlin_generalized:coeff}
\end{enumerate}
\setcounter{hypcount}{\value{enumi}}
\end{enumerate}
Note that while \ref{hyp:proof:ssmoothf} is stronger than the more general assumption \ref{hyp:nonlin_generalized:smoothapprox}, it allows us to obtain a quantitative estimate of the error rate in the main theorem (Thm.~\ref{thm:rank1}) of this paper, while this rate will be lost in the more general setting of  \ref{hyp:nonlin_generalized:smoothapprox}. 
\jump
Assumption \ref{hyp:nonlin_generalized:smoothapprox} is the most general form of our assumption on $f$ which does not need the non-linearity $f$ to be smooth. It is, in particular, satisfied - and this is proven in Sec.~\ref{sec:proofsmoothdensity} - in the critical case where $f$ satisfies \ref{hyp:nonlin_generalized}  and \emph{the noise distribution is sufficiently smooth and bounded}, that is, if we have the following joint assumptions on $f$ and $\mu_Z$.
\begin{enumerate}[leftmargin=1.5cm,label={(\bfseries H\arabic*)}]
\setcounter{enumi}{\value{hypcount}}
\item
\label{hyp:smoothZ+BC}
    \begin{enumerate}
        \item\label{hyp:smoothZ}
        The distribution $\mu_Z$ admits a density $w_Z:= \frac{\dd \mu_Z}{\dd z}$, with support on the whole real line.  Moreover, there exists an integer number $\ell$ such that  its $k$-th derivative $w_Z^{(k)}$ exists almost everywhere for all $k \leq \ell  + 1$ and satisfies $ w_Z^{(k)} \in L^2(\dd x) \cap L^\infty(\dd x) $, for all $k \leq \ell $. 
        \item\label{hyp:BC_fandwz} $f$ is a measurable function such that
        $\int |f(x)  w^{(k)}_Z(x) | \, \dd x < \infty $ for $k\leq \ell+1$.\setcounter{hypcount}{\value{enumi}}
        \item
        We then set, for $k\le \ell$,
        \begin{align}
\label{eq:informationcoeff_smoothdensity}
    \vartheta_k(f) &:= (-1)^{k} \int_{\mathbb{R}} f(x) w_Z^{(k)}(x) \dd x \, . &&
\end{align}
        and we assume that $\ell$ is larger than $\ks$, the smallest index $k$ so that $\vartheta_k(f)$ does not vanish. In this case, we define the information index of $f$ to be this index $\ks$.  
        
    \end{enumerate}
\end{enumerate}
Note that whenever $f$ is sufficiently smooth, Eq.\ \eqref{eq:informationcoeff_smoothdensity} simply corresponds to the integration by part of Eq.\ \eqref{eq:coeff_f_smoothf}. 
\begin{rem}
   The motivating example of $f_1(x) = |x|$ and $Z$ is Gaussian in Figure~\ref{fig:abs_rescaled} satisfies Hypothesis~\ref{hyp:nonlin_generalized:smoothapprox} because the Gaussian noise satisfies Hypothesis~\ref{hyp:smoothZ+BC}. 
\end{rem}
\jump
We conclude this section on the model with two remarks relating the information exponent $\vartheta_k(f)$ to orthogonal polynomials. 
\begin{rem}[Gaussian Noise and Hermite coefficients] For Gaussian noise, the density $w_G(x)$ satisfies \ref{hyp:smoothZ+BC} and using Stein's identities, $ w^{(k)}_G(x) = (-1)^{k} \Hek(x) w_G(x)$, one can write the information coefficient in terms of the \emph{Hermite coefficients} of $f$, namely  $\vartheta_k(f) = \langle f, \Hek \rangle_{L^2(\mu_G)}$.  These coefficients appear naturally in several recent studies related to RMT.
\end{rem}

\begin{rem}[Information coefficients and orthogonal decomposition] Under additional assumptions on $f_t$ and $\mu_Z$, one can differentiate term by term the decomposition of $f_t$ in the orthogonal polynomials basis $\{\opk\}_{k \in \N}$ associated to $\mu_Z$, to write down its $k$-th information coefficient as $\vartheta_k(f_t) =
        \sum_{n=k}^{\infty} \frac{\< f_t, \op_n \>}{\< \op_n , \op_n \>} \E_Z \op_n^{(k)}(Z)
$ and by Hypothesis \ref{hyp:nonlin_generalized:L4}, this yields the following formula:
\begin{align}
    \label{eq:def:coeff_f}
        \vartheta_k(f) 
        &=
        \sum_{n=k}^{\infty} \frac{\< f, \op_n \>}{\< \op_n , \op_n \>} \E_Z \op_n^{(k)}(Z)
        \, . &&
\end{align} 
\end{rem}
\jump

\subsection{Reminder on the BBP Phase Transition for the Linear Model}
\label{sec:usualBBP}
\jump 
Before jumping to the main result of this paper, let us first recall the classical result for the linear model, as the result for the non-linear model will be similar.

\jump  We say that a matrix $\rdmmat{W}$ is a \textbf{centered and normalized Wigner matrix with bounded fourth moment} if up to the symmetry constraint, $W_{ij} = W_{ji}$, the entries $\sqrt{N} W_{ij}$ are independent and distributed according to a law $\pP(.)$ which is independent of $N$, with mean zero and variance one,  and admits a finite fourth moment.  For such matrices, it is well known that as $N \to \infty$ the limiting spectral distribution is given by the semi-circular distribution with support between $-2$ and $2$ (see for example \cite{AGZ}) and the operator norm converges to $2$ since the fourth moment is bounded \cite{BaiYin88LargestEigenval}.

\jump Similarly, if we introduce the two deterministic piecewise functions 
\begin{align}
\label{eq:def:leadeig_BBP}
    \mathsf{l}(\Tilde{\gamma},\sigma) 
    &:=
    2 \sigma \, \mathrm{sign}(\tilde\gamma) \,  \Ind{   |\Tilde{\gamma} | < \sigma} +  \left( \tilde{\gamma} + \frac{\sigma^2}{ \tilde{\gamma}} \right) \Ind{   | \Tilde{\gamma} | \geq \sigma }  
    &\mbox{and} && 
    \mathsf{m}(\Tilde{\gamma},\sigma) 
    &:=
    \left( 1 - \frac{\sigma^2}{\tilde{\gamma}^2} \right) \Ind{    |\Tilde{\gamma} | \geq \sigma }
    \, , &&
\end{align}
where  $ \Ind{   x \geq c}$ is the indicator function, that is $ \Ind{     x \geq c } = 1$ if $x \geq c$ and is null otherwise, we have the following result for the rank-one perturbation of such matrices:

\jump
\begin{theo}[\cite{Pch2005,renfrew13} Rank-one perturbation of Wigner matrices]\label{thm:rank1_lin}
Let $\rdmmat{W}$ be a \emph{centered and normalized Wigner matrix with  bounded fourth moment}, $\vect{u} \in \mathbb{R}^N$  such that  $ \| \vect{u} \|  \to 1$ as $N\to \infty$, $\sigma>0$ and $\tilde\gamma$ real, then if we denote by $\lambda_1$ and $\rdmvect{v}_1$ the leading eigenvalue and corresponding eigenvector of $\sigma \rdmmat{W} + \Tilde{\gamma} \vect{u} \vect{u}^{\top}$, we have:
\begin{align}
\label{eq:BBP:leadingeig}
    \lambda_1 
    &\xrightarrow[N\to \infty]{\as}
    \mathsf{l}(\Tilde{\gamma},\sigma)
    \, , && 
\end{align}
and
\begin{align}
\label{eq:BBP:overlap}
    \langle \rdmvect{v}_1, \vect{u} \rangle^2 
    &\xrightarrow[N\to \infty]{\as}
    \mathsf{m}(\Tilde{\gamma},\sigma)
    \, . && 
\end{align}
where $\mathsf{l}(.,.)$ and $\mathsf{m}(.,.)$ are defined in Eq.\ \eqref{eq:def:leadeig_BBP}.
\end{theo}

\jump
Note that, up to re-scaling of $\Tilde{\gamma}$ and $\lambda_1$, one can always set $\sigma =1$ without loss of generality. Yet, forecasting the result of the non-linear case of the next section, it will be convenient to consider the general setting $\sigma>0$. We refer the readers to the works of Capitaine et al. \cite{Capitaine12fluctuations} and Diaconu \cite{diaconu2022heavy} for, respectively, the study of the fluctuations around these deterministic limits and the case where the fourth moment is infinite. 
\subsection{Phase Transition for the Non-Linear Model}
\label{sec:phasetransition}
We are now ready to state the main result of the paper:

\begin{theo}[Rank-one Equivalence for the non-linear model]\label{thm:rank1}
	Suppose that $f$, $\mu_Z$ and $\pi_X$ satisfy hypotheses~\ref{hyp:x}, \ref{hyp:noise} and \ref{hyp:nonlin_generalized:smoothapprox}. Let $\ks$ be the generalized information index associated with $f$ and suppose $\gamma(N)$ satisfies 
 \begin{align*}
     \frac{Q(\log(N)) |\gamma(N)|^{ \ks + 1} }{N^{\frac{\ks}{2}}} &\to  0 &&
 \end{align*}
    for every polynomial $Q$.
     Then, for every $\epsilon > 0$ and all $N$ sufficiently large depending on $\epsilon$, with high probability we have the following decomposition:
\begin{align}
	\Yf 
    &\ed 
    {\small{\sqrt{\vartheta_0(f^2) - \vartheta_0(f)^2} 
    \,}} \rdmmat{W}+ 
    \mat{P} + \rdmmat{E} 
    \, , &&
\end{align}
	where
	\begin{enumerate}
		\item  $\rdmmat{W}$ is a centered and normalized Wigner matrix  with  bounded fourth moment,
		\item $\mat{P}$ is the rank-one matrix:
\begin{align}
  \label{eq:def:rk1matrix}
    \mat{P}
    &:=
    \frac{\gamma(N)^{\ks}}{N^{ \frac{\ks}{2} - \frac{1}{2} }}  \frac{  \vartheta_{\ks}(f) }{\ks !}\bigg[  \frac{\vect{x}^{\ks}}{\sqrt{N}} \bigg(\frac{ \vect{x}^{\ks}}{\sqrt{N}}\bigg)^\intercal  \bigg]
    \, , &&
\end{align}
		\item  $\rdmmat{E}$ is a symmetric matrix with operator norm bounded by $\epsilon$ with high probability.
	\end{enumerate}
	Moreover, $\rdmmat{W}$ and $ \mat{P}$ are independent. 
\end{theo}
\begin{rem}\label{rem:rank1}
    If $f$ is sufficiently smooth and satisfies \ref{hyp:proof:ssmoothf} then one can get explicitly the error rate $\normop{\rdmmat{E}} = O(\frac{Q(\log(N)) |\gamma(N)|^{ \ks + 1} }{N^{\frac{\ks}{2}}} )$ (see Lemma~\ref{thm:smoothBBP}). The function $Q(.)$ in this error rate is explicitly given in Sec.~\ref{sec:outlineproof} and only depends on the stretched exponential decay of the distributions $\pi_X$ and $\mu_Z$, and in particular can be set to be a constant function whenever the two have a finite support. 
\end{rem}
\jump 
Note that the matrix $\sqrt{\vartheta_0(f^2) - \vartheta_0(f)^2} \rdmmat{W}$ corresponds to the case without  rank-one perturbation $(\gamma(N) =0)$. Thus, Theorem \ref{thm:rank1} states that in the high-dimensional regime, the non-linear model effectively behaves as a \emph{linear model} with the original rank-one perturbation $\gamma(N) (\vect{x}/\sqrt{N}) (\vect{x}/\sqrt{N})^\top$ replaced by the matrix $\mat{P}$ of Eq.~\eqref{eq:def:rk1matrix}. This phenomenon, where one can effectively lift the non-linearity, has been observed in other matrix models, see for example Refs.~\cite{ElKarouiSpikeKernel,ElKaroui10Kernel,Louart18RMTtoNN,Peche19PenningtonWorah,piccolo21}.
\jump 
Lastly, let's remark that, unlike many standard results in RMT, this result is \emph{non-universal} by nature as the coefficients $\vartheta_k(f)$ depend on both the non-linearity function $f$ \emph{and} the noise distribution $\mu_Z$, for example replacing $\mu_Z$ by a different distribution $\mu_{Z'}$  might lead to a different scaling of the SNR (see Cor.~\ref{cor:nontrivialscaling} bellow). 
\jump 
As a consequence of this theorem, it follows that the correct scaling of the SNR $\gamma(N)$ is such that the non-zero eigenvalue of the matrix $\mat{P}$ of Eq.~\eqref{eq:def:rk1matrix} is of order one, which leads to the following result. 

\begin{cor}[Relevant Scaling]\label{cor:nontrivialscaling}
Under the same hypothesis as in Thm.~\ref{thm:rank1}, the relevant scaling of the SNR is given 
 \begin{align}
 \label{eq:nontrivialscaling}
     \gamma(N) 
     &= 
     \gamma_0 N^{\frac{1}{2} ( 1 - \frac{1}{\ks})} \, . &&
 \end{align}
Furthermore, if $f$ satisfies \ref{hyp:proof:ssmoothf} we have  $\|\rdmmat{E}\|_{\mathrm{op}} = O( Q(\log(N)) \, N^{- \frac{1}{2 \ks} } ) \to 0$ at this scale.
\end{cor}
Note that $\gamma(N)$ is independent of $N$ only if $\ks=1$.  In particular and as a sanity check if  $f$ is the identity map, one retrieves the usual BBP setting with  $\vect{u} \equiv \vect{x}/\sqrt{N m_2}$. Let's also remark that whenever the noise is symmetric ($Z \ed -Z$) and $f$ is even, one can easily check that $\vartheta_1(f) =0$ and hence $\ks \geq 2$, leading to a $N$-dependent scaling of the SNR for the rank-one perturbation to survive in the large dimensional regime.  Note also the bound for the error term in Cor.~\ref{cor:nontrivialscaling} gets significantly worse as one increases the critical index $\ks$, implying that, in practice, one needs to \emph{push to larger values of $N$} for the non-linear matrix $\Yf$ to effectively ‘look like' a rank-one perturbation of a Wigner matrix, see Fig.~\ref{fig:abs_rescaled} and Fig.~\ref{fig:He3_rescaled} for illustrations.
\jump 
As $\Yf$ behaves as the standard rank-one perturbation of a Wigner matrix, one can get the limiting position of the possible outlier and the overlap of the top eigenvector thanks to
Thm.\ \ref{thm:rank1_lin}. Note that $\vect{x}^\ks/\sqrt{N}$ is not normalized to have a unit norm but since the $x_i$ are i.i.d with stretched exponential tails, one has $\| \vect{x}^\ks \|^2/N \xrightarrow[N\to\infty]{\mathrm{a.s}} m_{2 \ks}$, leading to the following result.

\begin{cor}[The BBP Transition]\label{thm:BBP_nonlinear}
Under the same Hypothesis of Theorem~\ref{thm:rank1} and under the relevant scaling of Eq.~\eqref{eq:nontrivialscaling}, if we denote by $\lambda_1$ and $\rdmvect{v}_1$ the leading eigenvalue and eigenvector of $\Yf$, we have:
\begin{align}
\label{eq:BBPNL:leadingeig}
    \lambda_1 
    &\xrightarrow[N\to \infty]{\as} 
    \mathsf{l} \left( \gamma_0^{\ks} \frac{ \vartheta_\ks(f)}{\ks !} \, m_{2\ks} \, , \, \sqrt{\vartheta_0(f^2) - \vartheta_0(f)^2} \right)
    \, , && 
\end{align}
and
\begin{align}
\label{eq:BBPNL:overlap}
    \Bigg\langle \rdmvect{v}_1, \frac{\vect{x}^{\ks}}{\| \vect{x}^{\ks}\|} \Bigg\rangle^2 
    &\xrightarrow[N\to \infty]{\as}
    \mathsf{m}\left( \gamma_0^{\ks} \frac{ \vartheta_\ks(f)}{\ks !} \, m_{2\ks} \, , \, \sqrt{\vartheta_0(f^2) - \vartheta_0(f)^2} \right)  
    \, . && 
\end{align}
where $\mathsf{l}(.,.)$ and $\mathsf{m}(.,.)$ are defined in Eq.\ \eqref{eq:def:leadeig_BBP}.
\end{cor}
\jump 
After the phase transition, the leading eigenvector $\rdmvect{v}_1$ becomes partially aligned with  $\vect{x}^{\ks}$ and \emph{not} the original vector $\vect{x}$. In particular,  whenever $\ks$ is even,  even after the transition, the matrix $\rdmvect{v}_1 \rdmvect{v}_1^{\top}$ loses all information on the sign of all the entries of $x_i x_j$.

\jump
Let's conclude this section regarding the assumptions on the signal vector $\vect{x}$. One should expect to lift the assumption \ref{hyp:x}  to extend the result to the case where the entries of $\vect{x}$ are independent but not necessarily identically distributed,  for example only being identically distributed (and different from zero) inside each block $B_{n}(N)$ of a partition of $N$  conditioned to  $|B_n(N)|/N \to \rho_n $\footnote{note that one should replace $m_{2\ks}$ by the appropriate limit of $\| \vect{x}^{\ks} \|^2/N $ in this case.}. Similarly, one can hope to replace the global independence condition by a \emph{weakly dependence condition}, for example by replacing $\vect{x}/\sqrt{N}$ by a vector $\vect{\sigma}$ drawn uniformly over $\Sn$, since for large $N$, $\vect{\sigma}$ behaves as a standard Gaussian vector normalized by its norm. However, one \textbf{cannot} replace $\vect{x}/\sqrt{N}$ with \emph{any} vector $\vect{u}$ with a fixed norm. Indeed, a crucial condition for the result to hold is that the vector $\vect{x}/\sqrt{N}$ is  \emph{delocalized}, with each entry carrying approximately the same weight of order $O(N^{-1/2})$ and replacing $\vect{x}/\sqrt{N}$ by a \emph{localized} vector (say for example $\vect{e}_1 = (1,0,\dots,0)$) will lead to a different scaling.

\section{Variants and Generalizations}
\label{sec:variants_and_gen}
In this section, we extend our result for natural variants and generalizations of Thm.~\ref{thm:rank1}. 
\subsection{Rank-\texorpdfstring{$K$}{} Deformation of Wigner Matrices}
\label{sec:finite_rk_deformation}
A natural extension of the previous setting is to replace the rank-one perturbation inside the non-linearity of Eq.~\eqref{eq:model_entries} by a rank-$K$ matrix where $K$ is an integer independent of $N$, that is we consider the matrix $\rdmmat{Y}_K^{(f)}$ with entries $Y_{K,ij}$ given by:
\begin{align}
	\label{eq:model_entries_rkK}
	Y_{K,ij} 
    &:= 
    \frac{1}{\sqrt{N}} \bigg[  f \left( Z_{ij} + \sum_{l=1}^{K} \frac{\gamma_l(N)}{\sqrt{N}} x_{i,l} x_{j,l} \right) - \E f(Z_{ij}) \bigg] 
    \, . &&
\end{align}
For $l = 1,\dots,K$, we suppose the column vectors  independent and distributed according to $\vect{x}_{l}:=(x_{1,l},\dots, x_{N,l}) \sim  \pi_{X,l}^{\otimes N}(.)$ where the distributions $\pi_{X,l}$  might be different for different index $l$. Let's remark that, in general, the family of vectors $\{\vect{x}_l\}_{l=1,\dots,K}$ does not form an orthogonal basis\footnote{however if one further assumes one of the $K$ distributions (say $\pi_{X,1}(.)$ without loss of generality) to have zero mean then for any $l \neq 1$ we have $\langle \vect{x}_{1}/\sqrt{N} ,\vect{x}_{l}/\sqrt{N} \rangle =0 $ with high probability, such that two different normalized vectors are almost surely orthogonal in the large $N$ limit.}. In the following, Hyp.~\ref{hyp:x} has to be understood for each $\pi_{X,l}$. The result of this generalization is as follows.

\begin{theo}[Small-Rank Equivalence For Non-Linear Rank-$K$ Perturbation]\label{thm:rankK}
	Suppose that Hypothesis~\ref{hyp:x}, \ref{hyp:noise} and \ref{hyp:proof:ssmoothf} hold, then with high probability, the  matrix model $\rdmmat{Y}_K^{(f)}$ has the following decomposition:
\begin{align}
\label{eq:rkKdecomposition}
    \rdmmat{Y}_K^{(f)}
    &\ed 
    {\small{\sqrt{\vartheta_0(f^2) - \vartheta_0(f)^2} 
    \,}} \rdmmat{W}+ 
    \mat{P}_K + \rdmmat{E}
    \, , &&
\end{align}
	where
	\begin{enumerate}
		\item  $\rdmmat{W}$ is centered and normalized Wigner matrix with entries with a bounded fourth moment,
		\item $\mat{P}_K$ is the \emph{finite-rank} matrix:
  \begin{align}
  \label{eq:def:rkKmatrix}
    \mat{P}_K
    &:=
    \frac{  \vartheta_{\ks}(f) }{ N^{ \frac{\ks}{2}} \ks !}\bigg[ \sum_{l=1}^K \gamma_l(N)  \vect{x}_l \vect{x}_l^\intercal  \bigg]^{  \odot \ks}
    \, , &&
\end{align}
		\item  $\rdmmat{E}$ is a symmetric matrix with operator norm bounded by $\frac{Q(\log(N)) |\gamma(N)|^{ \ks + 1} }{N^{\frac{\ks}{2}}}$ with high probability, with $Q(.)$ bounded by a polynomial function and $\gamma(N) := \max_l |\gamma_l(N)| $.
	\end{enumerate}
	Furthermore, $\rdmmat{W}$ and $ \mat{P}_K$ are independent.
\end{theo}
\jump
As before, the relevant scaling of each SNR $\gamma_l(N)$ is given by $ \gamma_l(N) = O(N^{\frac{1}{2} ( 1 - \frac{1}{\ks})})$. 

\begin{rem}[rank of $\mat{P}_K$] 
For $K \geq 2$, if for each $l=1,\dots,K$, one has the relevant scaling $\gamma_l(N) =\gamma_{0,l} \, N^{\frac{1}{2} ( 1 - \frac{1}{\ks})}$, the rank of the matrix $\mat{P}_K$ is - unlike the case $K=1$ - always higher than the one of the original perturbation matrix. In general, the rank of $\mat{P}_K$ is given as the number of different terms in its expression, that is $\rank(\mat{P}_K) = \binom{K+\ks-1}{K-1} \geq K$. For example, in the case $K=\ks=2$ one has: 
\begin{align}
  \label{eq:rkKeq2matrix_dec}
        \mat{P}_{K=\ks=2}
        &=
        \frac{  \vartheta_{\ks}(f) }{\ks !}
         \left[
        \gamma_{0,1}^2 \frac{\vect{x}_1^{2}}{\sqrt{N}} \bigg(\frac{ \vect{x}_1^{2}}{\sqrt{N}}\bigg)^\intercal
        + 
        \gamma_{0,2}^2 \frac{\vect{x}_2^{2}}{\sqrt{N}} \bigg(\frac{ \vect{x}_2^{2}}{\sqrt{N}}\bigg)^\intercal
        + 2 \gamma_{0,1} \gamma_{0,2}
        \frac{\vect{x}_1 \odot \vect{x}_2 }{\sqrt{N}} \bigg( \frac{\vect{x}_1 \odot \vect{x}_2 }{\sqrt{N}}\bigg)^\intercal
        \right]
        \, , &&
\end{align}
which, in general, is of rank three and contains a ‘cross-term' of the form $\vect{x}_1 \odot \vect{x}_2$.
\end{rem} 

\subsection{Variance-Profile Wigner Matrices}
\label{sec:varianceprofile_case}
In this section, we consider a generalization of our main result to inhomogeneous matrices. We refer to the series of work \cite{ajanki_universality_2017,ajanki_singularities_2017,ajanki_stability_2019} for a study of the spectral property of these models when there is no spike. 

Let $\mat{\Delta} = (\Delta_{ij})_{i,j \leq N} \in \R^{N \times N}$ be a (sequence of) \textbf{block-constant variance profile matrix}, that is given $n \geq 1$, there exists a partition  of $[N]:=\{1,\dots,N\}$ given by
\begin{align}
    [N] &= \bigsqcup_{s =1}^n I_s &&
\end{align}
	such that the $\Delta_{i,j}$ are constant in the groups $I_s \times I_t$ for $s,t \in \{1,\dots,n\}$
	\begin{align}\label{defD}
		\Delta_{ij} &= \Delta_{st}, \quad \text{ for $i \in I_s$, $j \in I_t$} &&
	\end{align}
	and $(\Delta_{st})_{s,t\leq n}$ are independent of $N$. 
	Furthermore, the proportions of indices in each group converges in the limit
	\begin{align}\label{eq:nt}
		\frac{|I_s|}{N} &\to \rho_s \in (0,1) \quad \text{for all} \quad s \leq n. &&
	\end{align}
The decomposition of $\Yf$ in Theorem~\ref{thm:rank1} implies the following.

\begin{cor}[Small Rank Equivalence with Variance Profiles]
\label{cor:rank1variance}
	Under the same Hypothesis of Theorem~\ref{thm:rank1}  and if $\mat{\Delta}$ is a variance profile matrix, we have with high probability
 \begin{align}
 \label{eq:rk}
     \mat{\Delta} \odot \Yf 
     &=
     \mat{\Delta} \odot \left( {\small{\sqrt{\vartheta_0(f^2) - \vartheta_0(f)^2} \,}} \rdmmat{W} \right) + \mat{\Delta} \odot  \mat{P} + \rdmmat{E} 
     \, , &&
 \end{align}
	where $\rdmmat{W}$, $\mat{P}$ and $\rdmmat{E}$ are as in Theorem~\ref{thm:rank1}.
\end{cor}

\subsection{Rank-One Deformation of Rectangular Matrices}
\label{sec:rectangular_case}
In this section, we consider the following ‘rectangular' counterpart of the main result. 
\jump
 We say that a $(N \times M)$ matrix $\rdmmat{X}$ is an  \textbf{i.i.d standard random rectangular matrix} (with entries with bounded fourth moment) if all its entries are i.i.d from a centered distribution, independent of $N$, with variance one and bounded fourth moment. We recall that a \textbf{Wishart matrix} $\rdmmat{W}$ is given in terms of i.i.d standard random rectangular matrix by $\rdmmat{W}:= \left(  \frac{\rdmmat{X}}{\sqrt{M}} \right) \left( \frac{\rdmmat{X}}{\sqrt{M}} \right)^{\top}$ and in the double scaling where $N,M \to \infty$ with fixed aspect ratio $q=N/M$, the empirical distribution of $\rdmmat{W}$ converges to the \emph{Mar\v{c}enko-Pastur distribution} \cite{Marcenko1967distribution}. 
 \jump
 Similar to the rank-one deformation of Wigner matrices, the rank-one deformation of Wishart matrices of the form
 \begin{align}
 \label{eq:def:rk1Wishart}
     \rdmmat{W'} 
     &:= 
     \left( \sigma \frac{\rdmmat{X}}{\sqrt{M}} + \gamma \vect{u} \vect{v}^{\top} \right) \left( \sigma \frac{\rdmmat{X}}{\sqrt{M}}  + \gamma \vect{u} \vect{v}^{\top} \right)^{\top}
     \, , &&
 \end{align}
 has been studied extensively in the literature, and the behaviors of the leading eigenvalue and leading eigenvector of $\rdmmat{W'}$  also exhibit a phase transition depending on the value of the constants $\gamma$ and $q$, see for example Ref.~\cite{BBP,paul_asymptotics_2007}. 
 \jump
 Our goal is to relate the natural non-linear extension of this model, to the usual ‘linear' setting. Namely, we consider the matrix
\begin{align}
\label{eq:def:model_rect}
	\Yf
    &:=
    \frac{1}{M} \left( f[ \rdmmat{Z} + \frac{\gamma(N)}{\sqrt{N}} \vect{u} \vect{v}^{\top} ] -\E f(Z) \mat{J}_{(N,M)} \right) \left( f[ \rdmmat{Z} + \frac{\gamma(N)}{\sqrt{N}} \vect{u} \vect{v}^{\top}  ] -\E f(Z) \mat{J}_{(N,M)} \right)^{\top} 
    \, , & 
\end{align}
where the function $f$ is applied entry-wise, $\rdmmat{Z} = (Z_{ij})_{1\leq i \leq N, 1 \leq j \leq M}$ with $Z_{ij} \iid \mu_Z$, $\vect{u}=(u_1,\dots,u_N)$ and $\vect{v}=(v_1,\dots,v_M)$    with $u_i \iid \pi_{U} $ and  $ v_i \iid \pi_{V}$ and $\mat{J}_{(N,M)}$ is  the  matrix of size $(N \times M)$ with entries equal to one. We assume the same set of assumptions as in Sec.~\ \ref{sec:model_and_assumptions} (where Hyp.\ \ref{hyp:x} has to be understood for both the distribution $\pi_U$ and $\pi_V$ and  $\rdmmat{Z}$ is not symmetric anymore), we also assume $N \leq M$ without loss of generality. To study the eigenvalues of $\Yf$,  we consider the  symmetrized matrix $\overline{\Yf} \in \R^{(N + M) \times (N + M)}$ defined by
\begin{align}
\label{eq:symmetrized_model_rect}
    \overline{\Yf} 
    &:= 
    \frac{1}{\sqrt{M}} 
    \begin{bmatrix}
	   \mat{0} & f[ \rdmmat{Z} + \frac{\gamma(N)}{\sqrt{N}} \vect{u} \vect{v}^{\top} ] -\E f(Z) \mat{J}_{(N,M)}    \\
	   f[ \rdmmat{Z} + \frac{\gamma(N)}{\sqrt{N}} \vect{u} \vect{v}^{\top} ]^{\top}  -\E f(Z) \mat{J}_{(M,N)} 
 & \mat{0} 
    \end{bmatrix}
    , &&
\end{align}
The matrix $\overline{\Yf}$ has $2N$ non-zero eigenvalues coming by pair $\{ +\lambda_i , -\lambda_i \}_{1 \leq i \leq N}$ where the $\lambda_i$'s are given as the square root of the $N$ eigenvalues of $\Yf$\footnote{the reader not familiar with this classical result may find it by simply looking at the characteristic polynomial of $\overline{\Yf}$ and express it in terms of the one of $\Yf$.}. One can express this new  matrix $\overline{\Yf}$ as 
\begin{align}
    \overline{\Yf} 
    &= 
    \mat{\Delta} \odot \tilde{\rdmmat{Y}}_f
    \quad \mbox{with} \, \,
    \mat{\Delta}:= \sqrt{1+\frac{N}{M}}
    \begin{bmatrix}
	   \rdmmat{0} & \mat{J}_{(N,M)} \\	
	   \mat{J}_{(M,N)}  & \rdmmat{0}
    \end{bmatrix}  &&
    \\  &\mbox{and} \, \,
    \tilde{\rdmmat{Y}}_f 
    := \frac{1}{\sqrt{N+M}} \left(
    f[ \Tilde{\rdmmat{Z}} + \frac{\tilde{\gamma}(N+M)}{\sqrt{N+M}} \vect{\tilde{x}} \vect{\tilde{x}}^{\top} ] -\E f(Z) \mat{J}_{(N+M,N+M)} \right) 
    , &&
\end{align}
where 
$\rdmmat{\Tilde{Z}}$ is now a $((N+M) \times (N+M))$ \emph{symmetric} matrix with iid entries distributed according to $\mu_Z$, $\Tilde{\gamma}(N+M):= \gamma(N) \sqrt{1+M/N}$, $\vect{\tilde{x}} := [ \vect{u},\vect{v}] \in \R^{N+M}$. Note that even though the entries of $\vect{\tilde{x}}$ are only identically distributed inside each block $B_1 := (1,\dots,N)$ and $B_2  := (N+1,\dots,N+M)$ but not globally, a quick check of the proof indicates that Corollary~\ref{cor:rank1variance} still holds. As a result, one can get the asymptotic decomposition of $\overline{\Yf}$, and hence of 
$\Yf$, as the spectrum of the two matrices are in one-to-one correspondence, up to the trivial $M-N$ zero eigenvalues of $\overline{\Yf}$. Eventually the factors $\sqrt{1+N/M}$ cancel out and one gets the following result for the original matrix $\Yf$.

\begin{cor}[Rank 1 Equivalence for Rectangular Matrices]\label{cor:rectangularmatrix}
Assume that Hypotheses~\ref{hyp:x}, \ref{hyp:noise}, and \ref{hyp:proof:ssmoothf}  hold, for $\Yf$ given by Eq.~\eqref{eq:def:model_rect}. Then, in the scaling limit where $M(N)=\lfloor N/q \rfloor $ with fixed aspect ratio $q$, we have with high probability the following decomposition:
\begin{align}
    \Yf
    &\ed 
    \left( \sqrt{\vartheta_0(f^2) - \vartheta_0(f)^2} \frac{\rdmmat{X}}{\sqrt{M}} + \mat{P} \right) \left( \sqrt{\vartheta_0(f^2) - \vartheta_0(f)^2}  \frac{\rdmmat{X}}{\sqrt{M}} + \mat{P} \right)^{\top} + \rdmmat{E} 
    \, , && 
\end{align}

	where
	\begin{enumerate}
		\item  $\rdmmat{X}$ is i.i.d standard random rectangular matrix,
		\item $\mat{P}$ is the rank-one matrix
		\begin{align}
		    \mat{P}
            &=
            \frac{\gamma(N)^{k_\star}}{N^{ \frac{k_\star}{2} - \frac{1}{2} }}  \frac{ \vartheta_{\ks}(f) }{k_\star !}\bigg[  \frac{\vect{u}^\ks}{\sqrt{N}} \bigg(\frac{ \vect{v}^\ks}{\sqrt{M}}\bigg)^\intercal  \bigg]
            \, , &&
		\end{align}
		\item  $\rdmmat{E}$ is a symmetric matrix with operator norm bounded by $\frac{Q(\log(N)) |\gamma(N)|^{ \ks + 1} }{N^{\frac{\ks}{2}}} $ with high probability.
	\end{enumerate}
	Furthermore, $\rdmmat{X}$ and  $\mat{P}$ are independent. In particular, it follows that the perturbation $\rdmmat{P}$ is of non-trivial order precisely when $\gamma = O(N^{\frac{1}{2} \left( 1 - \frac{1}{k_{\star}}\right)})$ and $\|\rdmmat{E}\|_{\mathrm{op}} = N^{- \frac{1}{ 2k_\star} }  \to 0$ at this scale.
\end{cor}

\begin{rem}
    Corollary~\ref{cor:rectangularmatrix} can be generalized to non-smooth $f$ satisfying Hypothesis~\ref{hyp:nonlin_generalized} with a sequence of smooth approximations \ref{hyp:nonlin_generalized:smoothapprox}. We will lose the explicit rate on $\normop{\rdmmat{E}}$ as in Theorem~\ref{thm:rank1}.
\end{rem}

\section{Outline of the Proof}
\label{sec:outlineproof}
The main idea of the proof is to notice that $\frac{\gamma(N)}{\sqrt{N}} x_i x_j $ is of order $\frac{\gamma(N)}{\sqrt{N}} $ and hence if $\gamma(N)$ does not grow too fast with $N$, this term is small compared to $Z_{ij} \approx O(1)$, such that one can hope to perform and control the expansion of $ f \left( Z_{ij} + \frac{\gamma(N)}{\sqrt{N}} x_i x_j \right)$  around $ f \left( Z_{ij} \right)$ in terms of powers of $x_i x_j$.
\jump
This is done in practice by comparing the result to the case where one adds additional smoothness and boundedness assumptions to the model for which such expansion can be easily controlled and then one obtains the general case by approximations of the former case. 
\jump
Precisely, we first obtain in Sec.~\ref{sec:proof:smooth} the main result of this paper in the case where we add to the assumptions \ref{hyp:x}, \ref{hyp:noise} and \ref{hyp:proof:ssmoothf} the following assumptions
\begin{enumerate}[leftmargin=1.5cm,label={(\bfseries{$\mathbf{\tilde{H}}$}\arabic*)}]
\setcounter{enumi}{\value{hypcount}}
\item\label{hyp:proof:bounded} 
The entries of $\vect{x}$ and $\rdmmat{Z}$ are bounded;
\setcounter{hypcount}{\value{enumi}}
\end{enumerate}
as a result of  Taylor's theorem.

Next, in Sec.~\ref{sec:proof:notsmooth} we detail how to remove the assumption \ref{hyp:proof:bounded} and replace the assumption \ref{hyp:proof:ssmoothf} by the more general assumption \ref{hyp:nonlin_generalized:smoothapprox} using concentration and tail inequalities. 

Eventually, in Sec.~\ref{sec:proofsmoothdensity}, we prove that if $f$ satisfies \ref{hyp:nonlin_generalized} and the noise satisfies the smooth conditions \ref{hyp:smoothZ+BC}, we can explicitly construct a sequence  $(f_t)_t$  of \emph{cut-smoothed-and-regularized} approximations of $f$ such that \ref{hyp:nonlin_generalized:smoothapprox} is satisfied.

\section{Analysis for Smooth Functions}
\label{sec:proof:smooth}

In this section, we describe the main intuition behind the derivation of the rank-one equivalence.

\begin{lem}[Rank-one Equivalence for the non-linear model]\label{thm:smoothBBP}
	Suppose that Hypotheses~\ref{hyp:x}, \ref{hyp:noise}, \ref{hyp:proof:ssmoothf} and \ref{hyp:proof:bounded}   hold, then with high probability we have the following decomposition:
\begin{align}
\label{eq:rk1decomposition}
	\Yf 
    &\ed 
    {\small{\sqrt{\vartheta_0(f^2) - \vartheta_0(f)^2} 
    \,}} \rdmmat{W}+ 
    \mat{P} + \rdmmat{E} 
    \, , &&
\end{align}
	where $\rdmmat{W}, \mat{P}, \rdmmat{E}$ are as in Thm.~\ref{thm:rank1}. Furthermore, for $Q(x) = x^{2k_\star + 2}$, $\rdmmat{E}$ has an operator norm bounded by 
 \begin{align}
      \frac{ Q(\log(N))}{{N^{\frac{1}{2}}}} \, |\gamma(N)|^{1  + \frac{1}{\ks} }
 \end{align}
 with high probability.
\end{lem}

\begin{proof}
	Under Hypothesis~\ref{hyp:proof:bounded}, there exists a constant $L \geq 1$ such that $\max_{1 \leq i \leq N} |x_i| \leq L$ almost surely. 
 Although it is not needed in this proof, we will write our error terms explicitly with respect to the constant $L$, which will be useful in future computations. 
	
	We start by doing Taylor expansion of our function around $Z_{ij}$, since we assume $f$ is $C^{\ks+1}$ in \ref{hyp:proof:ssmoothf}, 
	\begin{align}
		Y_{ij}
        &=
        \frac{1}{\sqrt{N}} ( f( Z_{ij} ) - \E_Z f (Z) ) + \sum_{k = 1}^{k_\star} \frac{1}{\sqrt{N}} \frac{f^{(k)}(Z_{ij})}{k!}\bigg[  \frac{\gamma(N)}{\sqrt{N}} x_i x_j \bigg]^k + \frac{f^{(k_\star + 1)}(\xi_{ij})}{(k_{\star} + 1)!} \bigg[  \frac{\gamma(N)}{\sqrt{N}} x_i x_j  \bigg]^{k_\star + 1}
        \, , \label{eq:taylor_expand_f}&& 
	\end{align}
	where $\xi_{ij} \in [ Z_{ij} - | \frac{\gamma}{\sqrt{N}} x_i x_j |, Z_{ij} + |\frac{\gamma}{\sqrt{N}} x_i x_j | ]$. We now approximate each term in the Taylor expansion
	\begin{enumerate}
		\item \text{\textbf{Constant Term}:} 	The first term leads to the matrix
		\begin{align}
		  \frac{1}{\sqrt{N}} \left[ f(\rdmmat{Z}) - E_Z f(Z) \mat{I} \right]
             &:=
            \frac{1}{\sqrt{N}}  \left[( f(Z_{ij}) - \E_Z f (Z)  ) \right]_{1 \leq i,j \leq N} 
            \, , &&
		\end{align}
		where the entries $( f(Z_{ij}) - \E_Z f (Z) )$ are i.i.d (up to the symmetry), of mean zero, variance $\varfZ$ and have a bounded fourth moment since $f \in L^4(\mu_Z)$. This term gives rise to the scaled matrix $\rdmmat{W}$ in Theorem \ref{thm:rank1}. 
  
		\item \textbf{Sub Critical Terms $(k < k_\star)$:} The non-critical terms admit the following decomposition
		\begin{align}
         \label{eq:subcriticalterm}
		  \frac{\gamma(N)^{k}}{N^{\frac{k}{2} - \frac{1}{2}}} \bigg[ \frac{f^{(k)}(Z_{ij})}{k!}\bigg[  \frac{x_i^{k}}{\sqrt{N}}\frac{ x_j^{k}}{\sqrt{N}}  \bigg] -  \frac{\E_Z f^{(k)}(Z)}{k!}\bigg[  \frac{x_i^{k}}{\sqrt{N}}\frac{ x_j^{k}}{\sqrt{N}}  \bigg] +   \frac{\E_Z f^{(k)}(Z)}{k!}\bigg[  \frac{x_i^{k}}{\sqrt{N}}\frac{ x_j^{k}}{\sqrt{N}}  \bigg] \bigg]
		\end{align}
		where the first term can be controlled using the fact that the matrix
  	\begin{align}
		  \frac{1}{\sqrt{N} k!} \left[ f^{(k)}(\rdmmat{Z}) - E_Z f^{(k)}(Z) \mat{I} \right]
            &:=
            \frac{1}{\sqrt{N}}  \left[ \frac{1}{k!} ( f^{(k)}(Z_{ij}) - \E_Z f^{(k)} (Z) ) \right]_{1 \leq i,j \leq N} 
            \, , &&
		\end{align}
		is again a Wigner matrix, which has a bounded operator norm since $f^{(k)} \in L^4(\mu_Z)$ thanks to the additional assumption \ref{hyp:proof:ssmoothf}. Since the entries of $x$ are uniformly bounded thanks to the additional assumption \ref{hyp:proof:bounded}, the diagonal matrix $\diag( \xk ) := \mathrm{Diag}(x_1^k, \dots, x_N^k)$ is also bounded in operator norm which leads to the following bound:
		\begin{align}
        \label{eq:wignerestimate}
			&\frac{\gamma(N)^{k}}{N^{\frac{k}{2}} k!}  \bigg\|  \bigg[  \frac{f^{(k)}(Z_{ij})}{\sqrt{N}}\bigg[  x_i^{k}x_j^{k}  \bigg] -   \frac{\E_Z f^{(k)}(Z)}{ \sqrt{N}}\bigg[  x_i^{k}x_j^{k}  \bigg] \bigg]_{ij} \bigg\|_{\mathrm{op}}  \notag
			\\&= 
            \frac{\gamma(N)^{k}}{N^{\frac{k}{2}} k!} \bigg\| \diag(\xk) \bigg(  \frac{f^{(k)}(\rdmmat{Z})}{ \sqrt{N}} -   \frac{ \E_Z f^{(k)}(Z)}{ \sqrt{N}} \bigg) \diag( \xk ) \bigg\|_{\mathrm{op}} \leq O\left( \frac{Q(L) |\gamma(N)|^{k}}{N^{\frac{k}{2}}} \right) 
            \, , &&
		\end{align}
		with high probability. Hence the matrix coming from the first term in \eqref{eq:subcriticalterm} can be absorbed in the error matrix $\rdmmat{E}$.
  We recall that if $f$ is smooth, we have $\coeffk = \E_Z f^{(k)}(Z)$ which is null by definition of $\ks$ and since $k < \ks$. Hence the second term in \eqref{eq:subcriticalterm} vanishes. 
  
		\item \textbf{Critical Terms $k_\star$:} The reasoning for the critical term is identical: we have again Eq.~\eqref{eq:subcriticalterm}  with  $ k =\ks$, but now the remainder term is non-zero and  given by:
		\begin{align}
		  \frac{\gamma(N)^{k_\star}}{N^{\frac{k_\star}{2} - \frac{1}{2}}} \E_Z \frac{f^{(k_\star)}(Z)}{k_\star!}\bigg[  \frac{x_i^{k_\star}}{\sqrt{N}}\frac{ x_j^{k_\star}}{\sqrt{N}}  \bigg] \bigg]  
            &=
            \frac{\gamma(N)^{k_\star}}{N^{\frac{k_\star}{2} - \frac{1}{2}}}  \frac{\vartheta_{k_\star} (f)}{k_\star!}\bigg[  \frac{x_i^{k_\star}}{\sqrt{N}}\frac{ x_j^{k_\star}}{\sqrt{N}}  \bigg] 
            \, , &&
		\end{align}
		which is nothing else than the $(ij)$ entry of our matrix $\mat{P}$ in Theorem \ref{thm:rank1}. 
		\item \textbf{Remainder Terms:}  The last terms can be written in matrix form as
		\begin{align}
		  \rdmmat{R} 
            &:= 
            \frac{\gamma(N)^{\ks + 1}}{(\ks + 1)! N^{\frac{\ks}{2} + 1 }} \diag(\vect{x}^{\ks + 1})\Big(  f^{(\ks + 1)}(\rdmmat{\xi}) \Big) \diag(\vect{x}^{\ks + 1}) 
             \, . &&
		\end{align}
		By our assumption \ref{hyp:proof:ssmoothf}, we have $\| f^{(k_{\star} + 1)}\|_\infty \leq C$. We can now obtain a $L^2$ bound on the spectrum for some universal constant $\bar C$ that only depends on $C$
		\begin{align}
		 \normf{ \rdmmat{R}}^2 :=  \Tr(\rdmmat{R}^2) 
            &\leq
            \bar C  Q(L) ^2 N^2 \frac{ \gamma(N)^{2\ks + 2} }{ ((\ks + 1)! )^2 N^{\ks + 2 } }
            \, , &&
		\end{align}
		which goes to zero if $\frac{|\gamma(N)|^{2 \ks+ 2 }}{N^{\ks}} \to 0$. Since
		\begin{align}
		\normop{ \rdmmat{R}}
        &\leq 
        \normf{ \rdmmat{R}} \leq O \left(\frac{Q(L) |\gamma(N)|^{ \ks + 1} }{N^{\frac{\ks}{2}}} \right)
        \, , &&
		\end{align}
		we can absorb this into the error term $\rdmmat{E}$. 
	\end{enumerate}
	
\end{proof}

\section{Analysis for Functions Satisfying Hypothesis \ref{hyp:nonlin_generalized:smoothapprox} }
\label{sec:proof:notsmooth}
{
\color{red}

}

In this section, we generalize the result in Section~\ref{sec:proof:smooth} to general (not necessarily smooth) functions $f$ satisfying Hypothesis~\ref{hyp:nonlin_generalized} which admit a sequence of smooth approximations~\ref{hyp:nonlin_generalized:smoothapprox}.  We also remove the conditions that $Z$ and $x$ have bounded support in Hypothesis~\ref{hyp:proof:bounded}, but satisfy Hypothesis~\ref{hyp:x} and Hypothesis~\ref{hyp:noise}.

\subsection{Bounding the Tails}

We begin by weakening the assumptions in Hypothesis~\ref{hyp:proof:bounded}. 
Given $L > 1$, consider the following smooth cutoff function
\begin{align}
    \chi_L(x) 
    &= 
    \begin{cases}
	1 & |x| \leq L - 1\\
	0 & |x| > L
\end{cases} \, , &&
\end{align}
and interpolated such that $\chi_L(x)$ is smooth for $|x| \in [L-1,L]$. We can further assume that this smooth approximation satisfies, for some integer number $K$ large enough,
\begin{align}
  \sup_{0 \leq k \leq K} \| \chi^{(k)}_L \|_\infty 
  &<
  \infty 
  \, . &&
\end{align}
Such a cutoff function can be generated by taking a mollification of the indicator function. We define the following truncated versions of $f$ and $x$,
\begin{align}
    f_L(x) 
    &=
    f(x) \chi_L(x)
    \, , &&
\end{align}
and
\begin{align}
x^L_i 
&= x_i \Ind{ |x_i| \leq L} + L  \Ind{ |x_i| > L}
. &&
\end{align}
We define the following matrix 
\begin{align}
\Ytf
&=
\frac{1}{\sqrt{N}} \bigg[  f_L \left( Z_{ij} + \frac{\gamma(N)}{\sqrt{N}} x^L_i x^L_j \right) - \E f_L(Z_{ij}) \bigg] 
. &&
\end{align}
The main result of this section is the following estimate which states that the truncated matrix $\Ytf$ and $\Yf$ are equivalent with high probability. 
\begin{lem}[Truncation Lemma]\label{lem:logtruncation} 
	Suppose that Hypothesis~\ref{hyp:x} and Hypothesis~\ref{hyp:noise} and Hypothesis~\ref{hyp:nonlin_generalized} hold. Let $\gamma = O( (\log(N))^{2k_\star + 1}  N^{\frac{1}{2} \left( 1 - \frac{1}{k_{\star}}\right)})$. Assume that the exponents of the stretched exponential decay of $\pi_X$ and $\mu_Z$ are bounded below by some $\alpha>0$ and  assume that  $L \ge L(N) =  (\frac{3}{c}\log(N) )^{\frac{1}{\alpha}}$ for some small enough constant $c$. We have
	\begin{align*}
	\pP( \| \Yf - \Ytf \|_\infty \neq 0 ) 
 &\leq 
 (N + N^2) e^{-c L^\alpha}
 . &&
	\end{align*}
  goes to zero polynomially fast in $1/N$.
\end{lem}

\begin{proof}
	Notice that
	\begin{align*}
		\{ \| \Yf - \Ytf \|_\infty \neq 0 \} &= \{ (\Yf)_{ij} = (\Ytf)_{ij} ~ \forall i,j \}^C
		\\&= \bigcup_{i \leq N} \{  x_i \neq x_i^L \} \cup \bigcup_{i,j \leq N}  \{ f(Z_{ij} + \frac{\gamma(N)}{\sqrt{N}} x^L_ix_j^L ) \neq f_L(Z_{ij} + \frac{\gamma(N)}{\sqrt{N}} x^L_ix_j^L \}
	\end{align*}
	By the union bound, it follows that
	\begin{align}
		\pP( \| \Yf - \Ytf \|_\infty \neq 0 ) &\leq N \pP( x_1 \neq x_1^L  ) + N^2 \pP\bigg( f(Z_{12} + \frac{\gamma(N)}{\sqrt{N}} x^L_1x_2^L \neq f_L(Z_{12} + \frac{\gamma(N)}{\sqrt{N}} x^L_1x_2^L ) \bigg)
		\\&\leq N \pP( x_1 \geq L(N)  ) + N^2 \pP\bigg( |Z_{12}| \geq L(N) - \frac{\gamma(N) L(N)^2 }{\sqrt{N}}  \bigg).&&
	\end{align}
	Next, notice that $\frac{\gamma(N)L(N)^2}{\sqrt{N}} \leq \frac{L(N)}{2}$ for $N$ sufficiently large. By the stretched exponential bound on the tails in Hypothesis~\ref{hyp:x} and Hypothesis~\ref{hyp:noise}, it follows that
	\begin{align}
	    \pP( \| \Yf - \Ytf \|_\infty \neq 0 ) 
     &\leq  (N + N^2) e^{-c L(N)^\alpha}
     &&
	\end{align}
	which goes to zero for our choice of $L$. 
\end{proof}

In particular, moving forward it suffices to study the spectrum of $\Ytf$ instead of $\Yf$. 

\subsection{Smooth Approximations}

By Hypothesis~\ref{hyp:nonlin_generalized:L4} and Hypothesis~\ref{hyp:nonlin_generalized:coeff}, we can find a sequence of smooth $f_t$ converging to $f$ in $\LfourZ$ such that
\begin{align}
 \coeffkn 
 &\xrightarrow[t \to \infty]{}
 \coeffk   
 &&
\end{align}
for all $k \leq k_\star$. We can construct a perturbation of the smoothing function such that the critical exponent is preserved, instead of being arbitrarily close. We define
\begin{align}
	g(x) 
	&:=
	\sum_{l = 0}^{\ks - 1}  c_l(f_t) x^l
	\, , &&
\end{align}
where the coefficients $c_l(f_t)$ are chosen to satisfy
\begin{align}
	\E_Z g^{(k)}(Z) 
	&=
	\sum_{l = k}^{\ks - 1} \frac{l !}{(l - k)!}  c_l(f_t) \E[ Z^{l - k} ] = - \coeffkn
	\, , &&
\end{align}
for all $k \leq \ks-1$. Such coefficients can be found recursively by setting 
\begin{align}
	c_{\ks - 1}(f_t) 
	&=
	\frac{-1}{(\ks - 1)!} \vartheta_{\ks - 1}(f_t)
	\, ,\nonumber \\
 c_{k}(f_t)&=\frac{1}{k!} \bigg(-\vartheta_{k}(f_t)-\sum_{l = k + 1}^{\ks -1} \frac{l!}{(l - k)!} c_l(f_t) \E[ Z^{l - k} ] \bigg)\, \text{ for } k < \ks - 1,&&\label{const}
\end{align}
(the system of equations we have to solve is an upper triangular matrix with non-zero diagonal so such coefficients exist). In particular, the constants $c_l(f_t)$ are continuous functions of the $\vartheta_k(f_t)$. It follows that
\begin{align}
	\fnp(x)
	& :=
	f_t(x) + g (x)
	\, , &&
\end{align}
preserves the critical exponent by linearity
\begin{align}
	\coeffkfnp = 0 \quad \for \quad k < k_\star 
	\qquad  \mbox{and} \qquad
	\coeffksfnp \neq 0 
	\, . &&
\end{align}
We now prove the main result of this subsection, which is an approximation theorem, which states that the spectrums of the smoothed random matrix $\Ytfnp$ is a good approximation for the spectrum of $\Ytf$.

\begin{lem}[Approximation Lemma]
	\label{lem:approxlemma}
Let $f$ satisfy Hypothesis~\ref{hyp:nonlin_generalized:cptsup} and 
Hypothesis~\ref{hyp:nonlin_generalized:L4} and suppose that $\gamma(N)$ is such that $\frac{\gamma(N)L(N)^2}{\sqrt{N}}  \to 0$. For every fixed $t \geq 0$ and $N$ sufficiently large depending on $N$, we have with high probability that
	\begin{align*}
		\normop{\Ytf - \Ytfnp} 
		 \leq 
		O(\| f - \fnp \|^2_{\LfourZ} + \| f - \fnp \|^4_{\LfourZ} ) \leq  o_t(1) + o_N(1)
		\, , &&
	\end{align*}
 where $\lim_{t \to \infty} o_t(1) = 0$ and $\lim_{t \to \infty} \lim_{N \to \infty} o_N(1)$.
\end{lem}

\begin{proof}
	We first fix an $x$ and do the computation conditionally on $x$. 
	\\\\
	\textit{Step 1:} Notice that $f_L$ is globally Lipschitz with constant $Q(L(N))$ and $\fn$ is globally Lipschitz with constant $\|\fn'\|_\infty < \infty$ because it is compactly supported and smooth by Hypothesis~\ref{hyp:nonlin_generalized:cptsup}. The operator norm is convex and a Lipschitz functions of its entries, so Talagrand's concentration inequality implies that
 \begin{align}
    \pP_Z( |	\normop{ \Ytf - \Ytfnp }  - \E_Z	\normop{ \Ytf - \Ytfnp } | \geq t ) &\leq
    e^{- \frac{ t^2N}{ C_t(L)^2}}.&& 
 \end{align}
	The constant $C_t(L) = C (Q(L(N)) + \|\fn'\|_\infty )^2$ is independent of $x$ and only depends on the Lipschitz constants of $f$ and $\fnp$, and is at most a polynomial in $L$. Therefore, the upperbound will tend to $0$ as $N \to \infty$. In particular, with high probability the operator norm is close to its expected value.
	\\\\
	\textit{Step 2:}  We define $\tilde f(Z) = f_L(Z + \frac{\gamma(N)}{\sqrt{N}} x_i x_j) $ and $\fnpt(Z) = \fnp(Z + \frac{\gamma(N)}{\sqrt{N}} x_i x_j) $ to be shifted functions. We will show that
	\begin{align}
		\label{eq:approxlatala}
		\E_Z \normop{\Ytf - \Ytfnp} 
		&\leq 
		O( \| \tilde f - \fnpt \|^4_{\LfourZ} + \| f_L - \fnp \|^4_{\LfourZ}) 
		\, . &&
	\end{align}
	Conditionally on $\vect{x}$, the matrix $ \Ytf - \Ytfnp$ has independent entries, we can apply Latala's Theorem (\cite{LATALA} or \cite[Lemma~6.3]{guionnet_fullLDP}) which bounds the expectation of the operator norm of a matrix with independent entries by the $L^2$ norm of its rows and the $L^4$ norm of its entries.
 Integrating with respect to the noise variables conditionally on $\vect{x}$ 
	\begin{align}
		&\frac{1}{N} \E \Bigg( f_L \left( Z_{ij} + \frac{\gamma(N)}{\sqrt{N}} x_i x_j \right) - \E f_L(Z_{ij})  - \fnp \left( Z_{ij} + \frac{\gamma(N)}{\sqrt{N}} x_i x_j \right) + \E\fnp (Z_{ij}) \Bigg)^2 
		\\
		&\leq 
		\frac{2}{N}  \E \Bigg( f_L \left( Z_{ij} + \frac{\gamma(N)}{\sqrt{N}} x_i x_j \right)  - \fnp \left( Z_{ij} + \frac{\gamma(N)}{\sqrt{N}} x_i x_j \right)  \Bigg)^2 + \frac{2}{N}  \E \Bigg(  \E f_L(Z_{ij})  - \E\fnp (Z_{ij}) \Bigg)^2
		\\
		&\leq 
		\frac{2}{N}  \| \tilde f_L - \fnpt \|^2_{\LtwoZ} + \frac{2}{N}   \|f_L -\fnp \|^2_{ \LoneZ}
		\, . &&
	\end{align}
	A similar computation for the fourth moment implies that
	\begin{align*}
		&\frac{1}{N^2} \E \Bigg( f_L \left( Z_{ij} + \frac{\gamma(N)}{\sqrt{N}} x_i x_j \right) - \E f_L(Z_{ij})  -  \fn \left( Z_{ij} + \frac{\gamma(N)}{\sqrt{N}} x_i x_j \right) - \E \fn (Z_{ij}) \Bigg)^4
		\\
		&\leq 
		\frac{8}{N^2}  \E \Bigg( f_L \left( Z_{ij} + \frac{\gamma(N)}{\sqrt{N}} x_i x_j \right)  -  \fn \left( Z_{ij} + \frac{\gamma(N)}{\sqrt{N}} x_i x_j \right)  \Bigg)^4 + \frac{8}{N^2}  \E \Bigg(  \E f_L(Z_{ij})  - \E \fn (Z_{ij}) \Bigg)^4
		\\
		&\leq
		\frac{8}{N^2}  \| \tilde f_L - \fnpt \|^4_{L^4(d\mu_{\mathcal{Z}})} + \frac{8}{N^2}   \|f_L - \fn \|^4_{ L^1(d\mu_{\mathcal{Z}}) }
		\, . &&
	\end{align*}
	Therefore, Latala's Theorem  implies \eqref{eq:approxlatala}. Since $\LfourZ \subseteq \LtwoZ \subseteq \LoneZ$ and our assumption on the integrability of $f$, it suffices to control the fourth moment which gives us \eqref{eq:approxlatala}.
	\\\\
	\textit{Step 3:} We now control the norm in the right hand side of \eqref{eq:approxlatala}. By the classical inequality
	\begin{align}
		(a + b)^4
		& \leq 
		2^4 (a^4 + b^4)
		\, , &&
	\end{align}
	there exists a universal constant $C_{k_\star}$  such that
	\begin{align}
		\| f_L -\fnp \|^4_{ \LfourZ } 
		&=
		C_{k_\star} \Big( \| f_L - \fn \|^4_{ \LfourZ } + \sum_{k < k_\star} c^4_k(\fn) \| x^k \|^4_{ \LfourZ } \Big)
		\, . &&
	\end{align}
	The second term is small order because $c^4_k(\fn)$ goes to zero as $\vartheta_\ell(f_t)$ go to zero for $\ell<\ks$ by \eqref{const} and $Z$ has finite moments by Hypothesis~\ref{hyp:noise}. 

 To control the first term, the triangle inequality implies that
 \begin{align}
     \| f_L -\fnp \|^4_{ \LfourZ } &\leq 2^4 \| f - f_L \|^4_{ \LfourZ } + 2^4 \| f - \fn \|^4_{ \LfourZ }. &&
 \end{align}
    The first term tends to $0$ as $N \to \infty$ while the second term goes to $0$ by Hypothesis~\ref{hyp:nonlin_generalized:L4}.

	Next, we need to control
 \begin{align}\label{eq:tildefbound}
     \| \tilde f_L - \tilde  \fnp \|^4_{ \LfourZ } &\leq 2^4 \| \tilde  f - \tilde  f_L \|^4_{ \LfourZ } + 2^4 \| \tilde  f - \tilde  \fn \|^4_{ \LfourZ } .&&
 \end{align}
 Notice that for any $\epsilon > 0$ and all $N$ sufficiently large
 $c = \frac{\gamma(N)}{\sqrt{N}} x^L_i x^L_j \in [-\epsilon,\epsilon]$ almost surely because $x_i^L x_j^L \in [ -L(N)^2, L(N)^2 ]$ so $|\frac{\gamma(N)}{\sqrt{N}} x^L_i x^L_j| \to 0$ for all $N$ sufficiently large by our Theorem assumptions. 
	Since $f$ is locally Lipschitz with polynomial constant, there exists a polynomial $P$ with bounded degree such that
 \begin{align*}
     \E_Z (\tilde f)^4 \leq 2^4 \E_Z f^4 + 2^4 \E (f(Z + c) - f(Z))^4 \leq  2^4 \E_Z f^4 + 2^4 c^4 \E P(|Z + c|) < \infty
 \end{align*}
 because $f \in L^4$ and Hypothesis~\ref{hyp:noise} implies that all moments of $Z$ are finite. This bound is also holds almost surely in $c$. Therefore, the first term of \eqref{eq:tildefbound} goes to zero. To control the second term, notice that 
 \begin{align}
    &\E \|  f(Z + c) -  \fn(Z + c) \|_{\LfourZ}^4 
    \\&\leq 3^4 ( \E \|  f(Z + c) -  f(Z) \|_{\LfourZ}^4 +  \E \|  \fn(Z + c) -  \fn(Z) \|_{\LfourZ}^4 + \E \|  f(Z) -  \fn(Z) \|_{\LfourZ}^4 )
 \end{align}
 The first and second term are clearly bounded by $O(\epsilon^4)$ because $f$ is locally Lipschitz with polynomial constant and $\fn$ is Lipschitz because it is smooth and compactly supported by Hypothesis~\ref{hyp:nonlin_generalized:cptsup}. The third term tends to zero by the convergence Hypothesis~\ref{hyp:nonlin_generalized:L4}.
\jump
	We, therefore, conclude that 
	\begin{align}
		\label{eq:L4bound}
		\| f -\fnp \|^4_{  \LfourZ  } 
		&= o_t(1) + o_N(1)
		\, , &&
	\end{align}
	uniformly for all $c$. All the bounds are held almost surely in $x$ and are independent of $x$, so our results hold unconditionally as well, so our smooth approximations are exact in the limit. 
\end{proof}

\subsection{A Smooth Approximation of the Main Theorem}

We now combine the arguments to prove Theorem~\ref{thm:rank1}.

\begin{proof}[Proof of Theorem~\ref{thm:rank1}]
    
By the truncation Lemma~\ref{lem:logtruncation} we have that with high probability that $\Yf$ and $\Ytf$ have the same spectrum. The approximation Lemma~\ref{lem:approxlemma}, we also have that with high probability, 
\begin{align}
    \Ytf 
    &\ed \Ytfnp + \rdmmat{E}_t + \rdmmat{E}_{(N)}
    \, , &&
\end{align}
where $\normop{\rdmmat{E_t}} \leq o_t(1)$,  $\normop{\rdmmat{E_{(N)}}} \leq o_N(1)$ when $N$ is sufficiently large depending on $t$. 

For every $t \geq 1$, the function $\fnp$ satisfies the conditions of Hypothesis~\ref{hyp:proof:ssmoothf} because of Hypothesis~\ref{hyp:nonlin_generalized:L4}. We can conclude that with high probability we have the following decomposition:
\begin{align}
	\Ytfnp 
    &\ed 
    {\small{\sqrt{\vartheta_0(f^2) - \vartheta_0(f)^2} 
    \,}} \rdmmat{W}+ 
    \mat{P} + \rdmmat{E}_{(N)} 
    \, . &&
\end{align}
We conclude by Lemma \ref{lem:logtruncation} that with high probability and all $N$ sufficiently large depending on $t$ that
\begin{align}
\Yf
    &\ed 
    {\small{\sqrt{\vartheta_0(f^2) - \vartheta_0(f)^2} 
    \,}} \rdmmat{W}+ 
    \mat{P} + \rdmmat{E}_{(N)} +  \rdmmat{E}_t
    \, &&
\end{align}
where the operator norm of $\normop{\rdmmat{E}_{(N)}} \to 0$ as $N \to \infty$ for any fixed $\epsilon$. In particular, given any $\epsilon > 0$, we may fix a $t$ such that $\normop{\rdmmat{E}_t} \leq \epsilon/2$ and then take $N$ sufficiently large depending on $\epsilon$ such that $\normop{\rdmmat{E}_{(N)}} < \epsilon / 2$ to conclude the proof.
\end{proof}  
\begin{rem}
    If we assume sufficient regularity on the function $f$ and are able to find a sequence of approximating $\fnp$ with uniformly bounded $\ks + 1$ derivatives for all $n$, then an explicit bound of the operator norm $\rdmmat{E}$ in terms of $N$ can be computed.
\end{rem}

We end this section with a remark that our proof naturally generalizes to the cases introduced in Section~\ref{sec:variants_and_gen}.

\begin{proof}[Proof of Theorem~\ref{thm:rankK}, Corollary~\ref{cor:rank1variance}, Corollary~\ref{cor:rectangularmatrix}] It is easy to see that the proof of Theorem~\ref{thm:rank1} also holds entry wise, so the proofs of the generalizations are identical to the simple case. We sketch the essential modifications. The critical observation is that the Taylor expansion done in Eq.~\eqref{eq:taylor_expand_f} holds if the function $f$ depends on the index, that is $f = f_{ij}$, and in the presence of a higher order spike. The bound in \eqref{eq:wignerestimate} also holds for $f_{ij}$  in the setting with a (bounded) variance profile by Latala's theorem. The same construction of smooth approximators also holds in this setting, so Section~\ref{sec:proof:notsmooth}.  

Corollary~\ref{cor:rectangularmatrix} naturally follows from Corollary~\ref{cor:rank1variance} and the Girko Hermitization trick in equation \eqref{eq:symmetrized_model_rect}.

\end{proof}

\section{Approximations in the Case of Smooth Density for the Noise}
\label{sec:proofsmoothdensity}

In this section, we show that under the assumptions that $f$ is locally Lipschitz and the noise $Z$ has a sufficiently smooth density given by Hypothesis~\ref{hyp:smoothZ+BC}, we explicitly construct a sequence $f_t$ satisfying Hypothesis~\ref{hyp:nonlin_generalized:smoothapprox}. 

\subsection{Smooth Truncation \texorpdfstring{$\fdm$}{}}

We now construct a sequence of smooth approximations of $f$ satisfying the Hypothesis~\ref{hyp:nonlin_generalized:smoothapprox}. We first truncate the domain of $f$ to a compact set and then smooth this truncated function. We will show that the smoothing can be made exact by taking the smoothing parameter to $0$ and then taking the truncation to infinity. 

For $M > 0$, we define the following truncation
\begin{align}
	\label{eq:def:fM}
	f_M(x) 
	&:= 
	f(x) \Ind{|x| \leq M} 
	\, . &&
\end{align}
Next, we consider a standard smoothing of the non-linearity. We define the mollifier $\eta_\delta \in C^\infty(\R)$ by
\begin{align}
	\label{eq:def:eta}
	\eta(x) 
	&:= 
	\frac{1}{\int_{-1}^1 \exp ( \frac{1}{ y^2 - 1 } ) \, \dd y } \exp \bigg( \frac{1}{ x^2 - 1 } \bigg) \Ind{ |x| <1 }
	\quad\text{and}\quad \eta_\delta(x) = \frac{1}{\delta} \eta\Big(\frac{x}{\delta}\Big).
	&&
\end{align}
Let $\fdm$ be its smoothing with a mollifier,
\begin{align}
	\label{eq:def:fdM}
	\fdm 
	&=
	\fm \ast \eta_\delta.
	&&
\end{align}
We recall the following standard facts about the mollifiers, see \cite[Proposition~4.18 and Theorem~4.22]{brezis_book}:
\begin{lem}
\label{lem:brezis}
	If $f\in \Lpdx$, then  $\fdm$ is a smooth function and $\fdm \to f_M$ in $\Lpdx$. Furthermore,
    \begin{align}
	\supp( \fdm  ) \subseteq \overline{ \supp( f_M) + \supp( \eta_\delta ) }
	&\subseteq
	{ [-M-\delta,M+\delta]}
	. &&
\end{align}
\end{lem}

Now suppose that the density $w_Z(z)$ of $Z$ satisfies Hypothesis~\ref{hyp:nonlin_generalized:smoothapprox}. We define
\begin{align}
	\coeffk
	&= (-1)^{k} \int f(x) \wkZ(x) \dd x
	\, . &&
\end{align}
We first prove that $\coeffkfdm$ converges to $\coeffk$. 

\begin{lem}
	\label{lem:phiapprox}
	If Hypothesis~\ref{hyp:nonlin_generalized} and Hypothesis~\ref{hyp:smoothZ+BC} holds, then 
	\begin{align}
		|\coeffk - \coeffkfdm | 
		&=
		o_M(\delta) + o_M(1) 
		\, . &&
	\end{align}
for all $k \leq k_\star < \infty$. The $o_M(\delta)$ term goes to $0$ as $\delta \to 0$ uniformly for all $M$.
\end{lem}

\begin{proof}
Since $\fdm$ is smooth, we can integrate by parts to conclude that
\begin{align}
    \coeffkfdm
          &= 
          \E_Z \fdm^{(k)}(Z) = (-1)^{k} \int \fdm(x) \wkZ(x) \dd x
          \, .  &&
\end{align}
In particular, 
\begin{align}
	| 	 \coeffkfdm - \coeffk|
	& =
	\bigg| \int (\fdm(x) - f(x)) \wkZ(x)  \dd x \bigg| \leq \int |\fdm(x) - f(x)| \, | \wkZ(x) |  \dd x
	\, .
\end{align}
	We decompose this integral into two regions
\begin{align}
	\begin{split}
		&\int |\fdm(x) - f(x)|  |\wkZ(x) |  \dd x 
		\\ &\quad =
		\int_{|x| \leq M} |\fdm(x) - f(x)|  |\wkZ(x) |  \dd x  + \int_{|x| > M} |\fdm(x) - f(x)|  |\wkZ(x) | \, \dd x  \,  . 
	\end{split}
	&&
\end{align}

On the first region, we can use Holder's inequality and use the assumption \ref{hyp:smoothZ} to see that
\begin{align}
	& \int_{|x| \leq M} |\fdm(x) - f(x)|  | \wkZ(x) | \, \dd x
	\\&\leq\bigg( \int |\fdm(x) - f(x)| ^2  \dd x \bigg)^{1/2} \bigg( \int_{|x| \leq M}   | \wkZ(x) |^2 \dd x \bigg)^{1/2}
	\\&\leq C_k \| \fdm(x) - f_M(x) \|_{\Ltwodx} = o_M(\delta)
	\, , &&
\end{align}
where there error goes to $0$ for every fixed $M$. On the second region, since $\fdm$ is supported on a subset of $[-M-\delta, M + \delta]$, we have
\begin{align}
	\begin{split}
		&\int_{|x| > M} |\fdm(x) - f(x)|  | \wkZ(x) | \, \dd x
		\\&\leq \int_{|x| > M + \delta} |\fdm(x) - f(x)|  | \wkZ(x) | \, \dd x +  \int_{M < |x| < M + \delta} |\fdm(x) - f(x)|  | \wkZ(x) | \, \dd x .
	\end{split}
	&&
\end{align}
The first term tends to $0$ as $M \to \infty$ uniformly for all $\delta$ since $f$ is integrable so the integral of its tail converges. For fixed $M$, we have using the assumption that $\wkZ$ is uniformly bounded
\begin{align}
	\sup_{M < |x| < M + \delta}  |\fdm(x) - f(x)|  | \wkZ(x) |
	&\leq 
	C(M)
	\, , &&
\end{align}
so the second integral vanishes if we send $\delta \to 0$. Therefore, we have
\begin{align}
	\lim_{M \to \infty} \lim_{\delta \to 0}	| \coeffkfdm - \coeffk | 
	&= 
	0
	\, . &&
\end{align}
\end{proof}
\jump
We now prove that there exists a parametrization $(\delta(t), M(t))$ such that $(\delta(t), M(t)) \to (0, \infty)$ and $f_t = f_{ \delta(t), M(t) }$ satisfies Hypothesis~\ref{hyp:nonlin_generalized:smoothapprox}.

\begin{prop}
    There exists a sequence $(\delta(t), M(t)) \to (0, \infty)$ such that $f_t = f_{ \delta(t), M(t) }$ satisfies Hypothesis~\ref{hyp:nonlin_generalized:smoothapprox}.
\end{prop}

\begin{proof}
    We prove each property separately. \jump
    \textit{Hypothesis~\ref{hyp:nonlin_generalized:L4}:} We have
	\begin{align}
		\| f - \fdm \|^4_{\LfourZ } 
		&=
		\E_Z [( f(Z) - \fdm (Z) )^4 \Ind{|Z| \leq M} ] +  \E_Z[ ( f(Z) - \fdm (Z) )^4  \Ind{|Z| > M} ]
		\, .  &&
	\end{align}
	On the set $|Z| \leq M$ the fact that the density $\frac{d \pP_Z}{dz} = w_Z$ is bounded by Hypothesis~\ref{hyp:noise} implies that
	\begin{align}
		\E_Z [( f(Z) - \fdm (Z) )^4 \Ind{|Z| \leq M} ] &= \E_Z [( f_M(Z) - \fdm (Z) )^4 \Ind{|Z| \leq M} ] 
		\\&\leq  
		\E_Z [( f_M(Z) - \fdm (Z) )^4  ] 
		\\&\leq
		C \| f_M - \fdm \|^4_{\Lfourdx} =  o_M(\delta)
		\, , &&
	\end{align}
	because $f_M$ is compactly supported so $\fdm \to f_M$ in $\Lfourdx$ . 
	
	On the set, $|Z| > M$, we have $\fdm (Z) = 0$ for $Z \in [-M - \delta, M + \delta]^c$ so using the Lipschitz property of $f$ 
	\begin{align}
		&\E_Z [( f(Z) - \fdm (Z) )^4 \Ind{|Z| > M} ] 
		\\&=
		\E_Z [( f(Z) - \fdm (Z) )^4 \Ind{M \leq |Z| \leq M + \delta)} ]  + \E_Z [( f(Z)  )^4 \Ind{|Z| \geq M + \delta} ] 
		\\&\leq
		\sup_{ M \leq x \leq M + \delta }  |f(x) - \fdm (x)| \pP( M \leq |Z| \leq M + \delta)  + o(M) = o_M(\delta) + o(M)
		\, , &&
	\end{align}
	where $o_M(\delta) \to 0$ as $\delta \to 0$ for every fixed $M$. In the third line, we used the fact that $f \in \LfourZ $ so its tails are integrable by the dominated convergence theorem and the fact that $f$ is Lipschitz. We can control our error by first sending $\delta \to 0$ then $M \to \infty$. In particular, we can jointly define a sequence by taking $\delta$ sufficiently small depending on $M$, which gives us our sequence $(\delta(t), M(t))$.
 \jump
 \textit{Hypothesis~\ref{hyp:nonlin_generalized:cptsup}:} This is immediate by Lemma~\ref{lem:brezis}.
 \jump
 \textit{Hypothesis~\ref{hyp:nonlin_generalized:derivbounds}:} We will prove that $\| f_{M,\delta}^{(k_* + 1)} \| \leq C(M,\delta)$ for some constant that does not depend on $N$. By Young's inequality and the fact that the perturbation term is a polynomial of degree less than $k_\star$, 
	\begin{align}
	\bigg\| \frac{\dd^{k_\star + 1}}{\dd x^{k^\star + 1}} \bigg( f_{M} \ast \eta_\delta + \sum_{l = 0}^{k^* - 1}  c_l(\fdm) x^l \bigg) \bigg\|_\infty 
 &\leq
 \| f_M \ast \eta^{(k_\star + 1)}_\delta \|_\infty \leq \|f_M\|_{L^1(\dd x)}  \| \eta^{(k_\star + 1)}_\delta \|_{\infty}
 . &&
	\end{align}
	We have
	\begin{align}
	 \|f_M\|_{L^1(\dd x)}  
  &\leq 
  2M \Big( \sup_{x \in [-M,M]} f(x) \Big)  
  &&
	\end{align}
	and
	\begin{align}
	\| \eta^{(k_\star + 1)}_\delta \|_{\infty} 
     &\leq 
    \frac{1}{\delta^{k_\star + 1}} \| \eta^{(k_\star + 1)} \|_\infty.
    &&
	\end{align}
	Both of these constants are independent of $N$, so we have control of the remainder term for free. 
	
	By the same argument, we also get that the fourth moment assumption is automatically satisfied for $k < k_\star$ because 
	\begin{align*}
		&\bigg\| \frac{\dd^{k}}{\dd x^{k}} \bigg( f_{M} \ast \eta_\delta + \sum_{l = 0}^{k^* - 1}  c_l(\fdm) x^l \bigg)  \bigg\|_{\LfourZ}
		\\&= \| f_M \ast \eta^{(k)}_\delta + \sum_{l \geq k}  c_l(\fdm) x^l  \|_{\LfourZ}
		\\&\leq C_{k_\star,k} \bigg( \|f^4_M\|_{L^1(\dd x)}  \| \eta^{(k)}_\delta \|_{\LfourZ} + \sum_{l \geq k}c_l(\fdm)  \|x^l \|_{\LfourZ} \bigg).
	\end{align*}
	This quantity can clearly be bounded independently of $N$, because we assumed control of the moments of $Z$ in Hypothesis~\ref{hyp:smoothZ+BC}. And in particular, the fourth moment condition is satisfied automatically if we assume sufficient integrability on $Z$.

 \jump \textit{Hypothesis~\ref{hyp:nonlin_generalized:coeff}:} This is immediate by Lemma~\ref{lem:phiapprox}. We can take a modify our sequence by possible taking $\delta$ even smaller depending on $M$ such that Hypothesis~\ref{hyp:nonlin_generalized:L4} is satisfied as well.
\end{proof}

\jump

\paragraph{Fundings and Acknowledgments -}
\label{sec:fund_and_acknow}
The work of A.G. and J.K. has received funding from the European Research Council (ERC) under the European Union
Horizon 2020 research and innovation program (grant agreement No. 884584), as well as from the Swiss National Science Foundation grantsgrants OperaGOST (grant number 200390) and SMArtNet (grant number 212049). The authors would like to thank Yunzhen Yao for bringing her work \cite{9517881} to their attention and for the initial numerical investigation of the spectrum for the absolute value non-linearity, and Aleksandr Pak for fruitful discussions at the early stage of this project. 

\bibliographystyle{amsplain}
\providecommand{\bysame}{\leavevmode\hbox to3em{\hrulefill}\thinspace}
\providecommand{\MR}{\relax\ifhmode\unskip\space\fi MR }
\providecommand{\MRhref}[2]{%
	\href{http://www.ams.org/mathscinet-getitem?mr=#1}{#2}
}
\providecommand{\href}[2]{#2}

\end{document}